\documentclass[11pt]{amsart}
\usepackage[cp1251]{inputenc}
\usepackage{color}
\usepackage{amsthm,amssymb,amsmath,latexsym}
\usepackage{graphicx}
\usepackage[T2A]{fontenc}
\usepackage{graphs}
\usepackage{MnSymbol}
\usepackage{bbm}
\usepackage{centernot}
\usepackage{mathrsfs}
\usepackage{amsthm}

\usepackage[colorlinks=true,linkcolor=blue, citecolor=blue]{hyperref}%

\title{Rigidity and Toeplitz systems}

\author{Henk Bruin}
\address{Faculty of Mathematics, University of Vienna, Oskar Morgensternplatz 1, Vienna, Austria}
\email{henk.bruin@univie.ac.at}

\author{Olena Karpel}
\address{AGH University of Krakow, Faculty of Applied Mathematics, al. Adama Mickiewicza~30, 30-059 Krak\'ow, Poland \&
B. Verkin Institute for Low Temperature Physics and Engineering,
47~Nauky Ave., Kharkiv, 61103, Ukraine}

\email{okarpel@agh.edu.pl}

\author{Piotr Oprocha}
\address{AGH University of Krakow, Faculty of Applied Mathematics, al. Adama Mickiewicza~30, 30-059 Krak\'ow, Poland \&
Centre of Excellence
IT4Innovations - Institute for Research and Applications of Fuzzy Modeling,
University of Ostrava, 30. Dubna 22, 701 03 Ostrava 1, Czech Republic
}

\email{oprocha@agh.edu.pl}

\author{Silvia Radinger}
\address{Faculty of Mathematics, University of Vienna, Oskar Morgensternplatz 1, Vienna, Austria}
\email{silvia.radinger@univie.ac.at }

\date{\today}
\usepackage{enumitem}

\usepackage{verbatim}
\usepackage{float}
\usepackage{amsmath}
\usepackage{graphicx}
\usepackage{mathtools}
\usepackage{amssymb}
\usepackage{amsthm}

\usepackage{geometry}

\theoremstyle{plain}
\newtheorem{thm}{\textsf{\textbf{Theorem}}}[section]
\newtheorem*{thm*}{\textsf{\textbf{Theorem}}}
\newtheorem{lem}[thm]{\textsf{\textbf{Lemma}}}
\newtheorem{cor}[thm]{\textsf{\textbf{Corollary}}}
\newtheorem{prop}[thm]{\emph{Proposition}}
\newtheorem*{prop*}{\emph{Proposition}}
\newtheorem*{claim*}{\emph{Claim}}

\usepackage{color}

\theoremstyle{definition}
\newtheorem{dfn}[thm]{\textbf{\textsf{Definition}}}

\newtheorem{rem}[thm]{{\textsf{Remark}}}
\newtheorem{ex}[thm]{{\textsf{Example}}}

\newtheorem{conj}[thm]{\emph{Conjecture}}

\newcommand{\cc}[1]{\mathcal #1}

\newcommand{\N}{\mathbb{N}}
\newcommand{\R}{\mathbb{R}}

\newcommand{\Z}{\mathbb{Z}}

\DeclareMathOperator{\dist}{dist}

\DeclareMathOperator{\id}{id}

\DeclareMathOperator{\orb}{orb}

\DeclareMathOperator{\len}{length}
\newcommand{\invlim}{\varprojlim}

\newcommand{\eps}{\varepsilon}

\usepackage{graphicx, tikz}

\def\ov{\overline}
\def\tl{\widetilde}
\newcommand{\wh}{\widehat}
\def\bx{\boldsymbol{x}}

\newcommand{\blue}{\color{blue}}
\newcommand{\red}{\color{red}}

\begin{document}
\begin{abstract}
		The aim of this paper is to study measure-theoretical rigidity and partial rigidity for classes of Cantor dynamical systems including Toeplitz systems and enumeration systems. 
   We use Bratteli diagrams to control invariant measures that are produced in our constructions. This leads to systems with desired properties. Among other things, we show that there exist Toeplitz systems with zero entropy which are not partially measure 
   -theoretically rigid with respect to any of its invariant measures. We investigate enumeration systems defined by a linear recursion, prove that all such systems are partially rigid and present an example of an enumeration system which is not measure-theoretically rigid. We construct a minimal $\mathcal{S}$-adic Toeplitz subshift which has countably infinitely many ergodic invariant probability measures which are rigid for the same rigidity sequence.
   \\[4mm]
%37B10 - Symbolic dynamics
%37A05 - Dynamical aspects of measure-preserving transformations
%37B20 - Notions of recurrence and recurrent behavior in topological dynamical systems
\end{abstract}

\maketitle
\tableofcontents

\section{Introduction}
This paper is devoted to the study of measure-theoretical rigidity for various classes of Cantor dynamical systems which include Toeplitz systems and enumeration systems.
Rigidity is a form of recurrence in dynamical systems: a finite measure-preserving dynamical system $(X,T,\mu)$ is called rigid if there is an increasing sequence $t_n \in \mathbb{N}$ such that $\mu(T^{-t_n}(A) \Delta A) \to 0$ for all measurable sets $A$ as $n \to \infty$. 
 The notion of rigidity was introduced for measure preserving transformations in \cite{FurstenbergWeiss1978}. Later, a weaker property of partial rigidity was introduced in \cite{Friedman1989}. 
 The notion of rigidity in topological dynamics
 (i.e., topological rigidity, also called uniform rigidity) was first considered in \cite{GM}. An overview of the results devoted to the rigidity sequences can be found in \cite{KanigowskiLemanczyk2023}. Recent results on rigidity include \cite{Bergelson-delJunco-Lemanczyk-Rosenblatt-2014}, \cite{FayadKanigowski2015}, \cite{Danilenko2016}, \cite{DonosoShao2017}, \cite{DonosoMaassRadic2023}. 

%Use two methods to determine measure-theoretical rigidity: a method from Downarowicz survey and measure extension from an odometer in Bratteli diagrams.

Toeplitz systems are minimal symbolic almost $1$-$1$ extensions of odometers which demonstrate rich variety of topological and measure-theoretical properties (see \cite{Downar} for a survey). 
Regular Toeplitz systems are measure-theoretically isomorphic to odometers, and thus are measure-theoretically rigid with respect to their unique invariant measure. Irregular Toeplitz systems satisfying the so called Same Aperiodic Readouts (SAR) property have a measure-theoretical representation as a skew product of their maximal equicontinuous factor (which is an odometer) and a subshift (see \cite{Downar}). Using this representation, we construct a Toeplitz subshift of zero entropy which is not partially measure-theoretically rigid with respect to any of its invariant measures.

Another tool that we apply in this paper is
Bratteli diagrams. These are graded graphs which are extensively used in Cantor dynamics for constructing models of homeomorphisms of Cantor spaces. Bratteli-Vershik systems proved to be very useful for classifying Cantor dynamical systems and constructing various examples of homeomorphisms (see e.g. \cite{Durand2010}, \cite{BezuglyiKarpel2016}, \cite{DownarowiczKarpel2018}, \cite{BezuglyiKarpel_2020}, \cite{DurandPerrin2022}). In particular, Bratteli diagrams are extremely useful in describing the simplex of invariant probability measures. In this paper, we show rigidity for the measures obtained as extensions from subdiagrams which are odometers. We apply the results to Toeplitz systems and some systems with countably infinitely many ergodic invariant probability measures. Enumeration systems are generalizations of odometers that are introduced in \cite{BaratDownarowiczIwanikLiardet2000, BaratDownarowiczLiardet2002}, see also \cite[Section 5.3]{Bruin_book}.
We show that a class of enumeration systems, coming from kneading theory, are  measure-theoretically partially rigid, and that there exists an enumeration system in this class (in fact, a substitution system) that is not measure-theoretically rigid. 

%we utilize the procedure of a measure extension from a vertex subdiagram which is studied in detail in \cite{BezuglyiKarpelKwiatkowski2015, AdamskaBezuglyiKarpelKwiatkowski2017}. 

The outline of the paper is as follows. In Section~\ref{Sect:Prelim}, we give preliminaries concerning rigidity, Bratteli diagrams and Toeplitz sequences. In Section~\ref{Sect:Toeplitz}, we construct a Toeplitz subshift of zero entropy which is not partially measure-theoretically rigid with respect to any of its invariant measures. In Section~\ref{Sect:ext_from_odom}, we show that a class of Bratteli-Vershik systems with the invariant probability measure which is obtained as an extension from an odometer with growing number of edges on each level is measure-theoretically rigid. Section~\ref{sec:enum_systems} gives examples of non-rigid partially rigid systems. We present a class of non-simple stationary Bratteli diagrams such that the corresponding Bratteli-Vershik map %non-rigid partially rigid 
is not measure-theoretically rigid but is 
partially measure-theoretically rigid
with respect to the full ergodic invariant measure. We also 
show that a large class of enumeration systems are partially measure-theoretically rigid, and that there exists an enumeration system which is not measure-theoretically rigid. In Section~\ref{Sect:OP}, we present a list of open problems that would be interesting to investigate further.
Our main results are contained in Theorems~\ref{thm:skew}, \ref{cor:toeplitz_non_part_rigidity}, \ref{Thm:rigid_mu_from_odom}, \ref{Thm:rigid_mu_from_odom_withoutorder}, \ref{thm:enum_syst_notrigid}.
%exist both meas theor rigid and not in enum syst

\medskip
\section{Preliminaries}\label{Sect:Prelim}
By a Cantor dynamical system we mean a pair $(X,T)$, where $X$ is a Cantor set and $T$ is a continuous surjective map.  We always consider the Borel $\sigma$-algebra $\mathcal{B}$ on $X$ and the Borel ergodic $T$-invariant probability measures. 
Throughout the paper, $\mathbb N, \Z, \R,$  $\mathbb{N}_0 = \N \cup \{0\}$ are the standard notations for the sets of numbers, and $|\cdot |$ denotes the cardinality of a set. 

\subsection{Rigidity in dynamical systems}

\begin{dfn}
    A dynamical system $(X,T, \mu)$ is \emph{measure-theoretically rigid} if there exists a sequence $t_n \to \infty$ such that
    $$
        \mu(T^{-t_n}(A) \Delta A) \to 0 \text{ for all measurable sets $A$.}
    $$
    The sequence $(t_n)_n$ is called a \textit{rigidity sequence}.
\end{dfn}

\begin{rem}\label{rem:minus}
 This convergence is not uniform. For example, if $(X,\mu,T)$ is the dyadic odometer, then $(2^n)_{n \geq 1}$ is a rigidity sequence, but if $A_n = \{ x \in X : x_{n+1}=1 \}$, then $\mu(A_n \triangle T^{2^n}(A_n)) \equiv 1$.
\end{rem}

\begin{rem}
This definition of measure-theoretical rigidity is equivalent to $\mu(T^{-t_n}A \cap A) \to \mu(A)$ as $n\to\infty$ for all measurable sets $A$.
\begin{align*}
    \mu(T^{-t_n}A \cap A) &= \mu(T^{-t_n}A \cup A) - \mu(T^{-t_n}A \Delta A) \\
            &= \mu(T^{-t_n}A) + \mu(A) - \mu(T^{-t_n}A \cap A) - \mu(T^{-t_n}A \Delta A)\\
\intertext{and thus}
   \lim_{n\to \infty} 2\mu(T^{-t_n}A \cap A) &= \lim_{n\to \infty} 2\mu(A) - \mu(T^{-t_n}A \Delta A) = 2\mu(A)        
\end{align*}
\end{rem}

%Here we always assume the dynamics $T$ to be a measure-preserving homeomorphism and the measure $\mu$ to be finite (for partial rigidity Silva \cite{Silva_book} also allows infinite measures).\\

%{\red What's the difference between this and definition?}

The following equivalent definition of rigidity can be found for instance in \cite{Silva_book}:
\begin{lem}
    A dynamical system $(X,T, \mu)$ is measure-theoretically rigid if and only if for all measurable sets $A$ and $\varepsilon > 0$ there exists an integer $n=n(\varepsilon) > 0$ such that 
    $$
    \mu(T^{-n}(A) \Delta A) < \varepsilon.
    $$
\end{lem}

\begin{rem}
Let $(X,T)$ be a surjective dynamical system and $\mu$ be a $T$-invariant non-atomic probability measure on $(X,T)$. 
Let $F$ be the set of all non-invertible points of $T$. Assume that $\mu(F) = 0$. Then as in the case of $T$ being a homeomorphism we can show rigidity by looking at the image of any measurable set $A$ under $T$ i.e.,
$$
    \mu(T^n(A) \Delta A)<\varepsilon.
$$

%For a dynamical system $(X,T)$ with a $T$-invariant, non-atomic measure $\mu$ and $T$ only having one non-injective point with $T^{-1}(x) = \{y^{(1)}, \dots, y^{(d)}\}$ we can show rigidity by looking at
Indeed, we know that 
$$
T^{-n}(T^n A) = \begin{cases}
    A \cup \bigcup_{j\leq n} T^{-n+j}(x) & \text{ if $T^j(x) \in T^nA$ for some $j\leq n$ and $x\in F$} \\
    A & \text{ otherwise.}
\end{cases}
$$
Suppose $\mu(T^nA \Delta A) < \varepsilon$, then 
\begin{align*}
\varepsilon &> \mu(T^{-n}(T^n A \Delta A)) = \mu(T^{-n}(T^n A) \Delta T^{-n}A) 
= \mu((A\cup \tl{F}) \Delta T^{-n}A)\\
&= \mu(A\setminus T^{-n} A) + \mu(\tl F\setminus T^{-n} A)+ \mu(T^{-n} A \setminus A) - \mu(\tl F\cap T^{-n} A \setminus A) \\
&= \mu(A \Delta T^{-n} A)
\end{align*}
for $\tl F$, a set of measure zero. Thus for such systems $\mu(T^{-n}A\Delta A) = \mu(T^{n}A\Delta A) < \varepsilon$ for all $n$ and $A$, therefore $(X,T)$ is rigid.
We will see non-invertible systems in Section~\ref{sec:enum_systems}, where we consider enumeration systems that have one point with multiple preimages. All measure-theoretically rigid systems are invertible a.e. (see Remark~\ref{rem:invertibility_part_rigid}).
\end{rem}

 \begin{lem}
  If $(X,d)$ is a compact metric space and $(X,\mu,T)$ is rigid, then for all $\eps > 0$ there is an $s \in \N$ such that
  $\mu( \{ x \in X : d(T^s(x),x) < \eps \} ) > 1-\eps$.
\end{lem}

\begin{proof}
Let $\eps > 0$ be arbitrary. Since $X$ is compact, there are $N \in \N$ and $x_i \in X$ such that $\{ A_i\}_{i=1}^N := \{ B(x_i ; \eps/2)\}_{i=1}^N$ is a cover of $X$. Let $s \in \N$ be such that
$\mu(A_i \triangle T^{-s}(A_i)) < \eps/N$ for all $1 \leq i \leq N$.
Then for each $x \in A_i \cap T^{-s}(A_i)$ we have $d(x,T^s(x)) < \eps$.
The complement has measure
$\mu( \{ x \in X : d(x, T^s(x)) \geq \eps\}  )
\leq \sum_{i=1}^N \mu(A_i \triangle T^s(A_i)) < N \frac{\eps}{N} < \eps$.
\end{proof}

\iffalse
\begin{rem}
 Assume that $T$ is invertible $\mu$-a.e.
and let $U_\eps(s) = \{ x \in X : d(x, T^s(x)) < \eps\}$.
For $\mu$-a.e.\ $x \in U_\eps(s)$ holds
$d(y, T^{-s}(y)) < \eps$ for $y = T^{\blue s}%{-s}
(x)$.
Therefore $\mu(U_\eps(-s)) \geq \mu(U_\eps)$ and also
$\mu(U_\eps(-s) \cap Y_\eps)
\geq \mu(U_\eps(s) \cap Y)$,
if $Y \subset X$ is measurable and $Y_\eps$ is its $\eps$-neighborhood.
\end{rem}

\begin{cor}\label{cor:minus}If $T$ is invertible $\mu$-a.e., then
$n$ is a rigidity time if and only if $-n$ is a rigidity time.
\end{cor}

\begin{proof}
 Clearly $\mu(A \triangle T^n(A)) = \mu( T^{-n}(A')\triangle A')$ for $A' = T^n(A)$.
 Since $T^n$ is an isometry on $(X,\mathcal{B},\mu)$, the corollary follows. {\blue [P: I think we need a more careful argument here. Rigidity is for sequences, not times, so it should read that $-t_n$ is rigidity sequence?]}
\end{proof}
\fi

A sufficient condition for rigidity (for ergodic measure-preserving continuous maps on first countable compact Hausdorff space) is that the ergodic measure $\mu$ has discrete spectrum.
This follows by the Halmos-Von Neumann Theorem \cite{HvN42}, which says that the system is isomorphic to a minimal rotation on a compact Abelian group $G$ with Haar measure.
The conditions on $X$ imply that $G$ is metrizable in a way that the group rotation is an isometry, and therefore rigid, see Lemma~\ref{lem:isometry}

\begin{dfn}
    A dynamical system $(X,T, \mu)$ is \emph{partially measure-theoretically rigid} if there exists a constant $\alpha>0$ and a sequence $s_n \to \infty$ such that
    \begin{equation}\label{Eq:part-rigid-def}
        \liminf_{n\to \infty} \mu(T^{-s_n}(A) \cap A) \geq \alpha\mu(A) \text{ for all measurable sets $A$.}
        %\text{ for sets $A$ of finite measure.}
    \end{equation}
\end{dfn}

%\tcr{Add about partial rigidity rate? And other results from \cite{DonosoMaassRadic2023}}

\begin{prop}\cite{Silva_book}\label{Prop:rigidity-cyl-sets}
    Let $T$ be a finite measure-preserving transformation satisfying~\eqref{Eq:part-rigid-def} for all sets $A$ is a dense algebra. Then $T$ is partially rigid along the same sequence $(s_n)$.
\end{prop}

\begin{rem}
Note that $T$ is rigid if and only if $T$ is partially rigid for $\alpha = 1$. Hence Proposition~\ref{Prop:rigidity-cyl-sets} is also true for rigid transformations. 
\end{rem}
\begin{prop}
    A finite measure-preserving transformation $T$ is rigid if and only if there is a sequence $\alpha_n \rightarrow 1$ such that $T$ is partially rigid with $\alpha_n$ for all $n$.
\end{prop}

\begin{rem}\label{Thm:rigidity_any_inv_mu_a_n_b_n}
In \cite[Remark 2.2]{DonosoMaassRadic2023} it is shown that if every ergodic invariant probability measure for a topological dynamical system is rigid (or partially rigid) with the same rigidity sequence $(t_n)_n$ then every invariant probability measure is rigid (or partially rigid) with the same rigidity sequence. It is also easy to see that if a topological dynamical system has countably many rigid ergodic invariant probability measures (not necessarily with the same rigidity sequence), then every invariant probability measure is partially rigid. This doesn't hold for measures for which the ergodic decomposition consists of uncountably many parts. A counterexample is
the twist map $T:(x,y) \mapsto (x,x+y \pmod 1)$ 
defined on the cylinder $[0,1] \times {\mathbb S}^1$ preserving Lebesgue measure, which is not partially rigid w.r.t.\ Lebesgue measure $\mu$. Indeed, for every $n \geq 1$ and $\eps > 0$, $\mu(\{ z \in [0,1] \times {\mathbb S}^1 : d(z, T^n(z)) < \eps\}) < 3\eps$.
If $\alpha$ was the partial rigidity constant, we take $\eps < \alpha/3$ and $A = (0,\eps) \times {\mathbb S^1}$. Then $\mu(T^n(A) \cap A) < 3\eps \mu(A) < \alpha\mu(A)$, making \eqref{Eq:part-rigid-def} impossible. This example is of course not ergodic, and hence not mixing,
but it satisfies a form of mixing known as Keplerian shear, see \cite{BS24}.
%Thus, if $(X_B, \varphi_B, \mu_i)$ are as in Theorem~\ref{Thm:rigid_a_n_b_n_2_meas}, then for any invariant probability measure $\mu$, the system $(X_B, \varphi_B, \mu)$ is measure-theoretically rigid.
\end{rem}

Clearly, measure-theoretical rigidity, partial rigidity and the corresponding rigidity sequences are preserved under measure-theoretical conjugacy.
%\begin{rem} 
%Measure-theoretical rigidity is preserved under measure-theoretical conjugacy. For two measurable dynamical systems $(X,\mathcal{A}, T, \mu)$ and $(Y, \mathcal{B}, S, \nu)$, let $g: X_0 \to Y_0$ be a bimeasurable bijection between invariant subsets $X_0 \subset X$, $Y_0\subset Y$ with $\mu(X_0) = \nu(Y_0) = 0$. If $(Y, \mathcal{B}, S, \nu)$ is measure-theoretically rigid with sequence $(n_k)_k$, it follows that $(X,\mathcal{A}, T, \mu)$ is measure-theoretically rigid with the same sequence as for all $A\in\mathcal{A}$ 
    %\begin{align*}
      %  \mu(T^{n_k}A \Delta A) &= \mu(T^{n_k}(g^{-1}B) \Delta g^{-1}B)\\
                  %  &=\mu(g^{-1}(S^{n_k}B \Delta B))\\
                  %  &= \nu(S^{n_k}B \Delta B) \rightarrow 0 \qquad \text{ as $k\to \infty$.}
   % \end{align*}      
%\end{rem}

%\cite{Bergelson-delJunco-Lemanczyk-Rosenblatt-2014}

\begin{lem}\cite{Silva_book}
    Let $T$ be a transformation on a non-atomic probability space $(X,\mathcal{S},\mu)$. If $T$ is partially rigid, then $T$ is not mixing. 
\end{lem}

    %{\color{blue}If it is regular Toeplitz system then yes, because it is measure-theoretically conjugate to an odometer which is its maximal equicontinuous factor \cite{Markley1975}.
    
    %In nonregular case the answer depends. There are some examples still measure-theoretically conjugate to an odometer (sometimes different than equicontinuous factor). Still it is rigid.
    
    %But there are also cases that Toeplitz is conjugate to skew product $G_s\times Y$
    %where $Y$ is full shift and $\sigma$ acts on second coordinate when $x$ returns to some set $D$ of positive measure. It seems that if $Y$ is e.g. full shift, then induced map $\tau_d\times \sigma$ can "expand" small set $A$ on second coordinate to the whole case and so huge part of points of $A$ never comes back to $A$. It must be done more formally, but it seems as good example for that case.
    
    %If the follow following does not have gap, then it covers all Toeplitz shifts with positive entropy. But there should be many other, like possibly all mixing readouts.

%A (two-sided) Toeplitz sequence $\omega$ can be viewed as a finite-valued semicocycle $\omega : \mathbb{Z} \rightarrow A$, where $A$ is the alphabet. Identify $\mathbb{Z}$ with a (two-sided) orbit of zero in $G_s$. In general, $\omega$ does not extend to a continuous function on all $G_s$. 

%\begin{thm}\cite{Downar}
%The map $\rho$ is a Borel$^*$ conjugacy between $(\underline{D}, \sigma_{\underline{D}})$ and $(\Omega, \tau_D \times \sigma)$. 
%\end{thm}

    \begin{thm}\label{thm:pmtr_ze}
        If $(X,T,\mu)$ has positive entropy, then it is not partially measure-theoretically rigid. 
    \end{thm}
    
    \begin{proof}
    Since $(X,T,\mu)$ has positive entropy, by Sina\u{\i}'s factor theorem, it factors onto a Bernoulli shift $(S^{\mathbb{Z}},\sigma,\nu)$ of the same entropy (see \cite[Theorem 12.7]{Denker}). It is clear that for any $\alpha$ there is a cylinder set that does not satisfy condition in the definition 
    $$\nu(\sigma^{n}(A) \cap A) \geq \alpha\nu(A)
    $$
    for any $n$ sufficiently large.
    Pulling back this set for $\mu$ completes the proof.
    \end{proof}

\begin{rem}\label{rem:invertibility_part_rigid}
%Every uniformly rigid system is clearly invertible. 
From Theorem~\ref{thm:pmtr_ze} it follows that all partially measure-theoretically rigid systems have zero entropy, and by \cite[Corollary~4.14.3]{Walters} every zero entropy probability measure preserving system on a compact metric space is invertible a.e. Thus, partially measure-theoretically rigid systems (on compact metric spaces) are invertible $\mu$-a.e.
%Similarly, partially measure-theoretically rigid systems (on compact metric spaces) are invertible a.e. 
%(it is enough to combine Theorem~\ref{thm:pmtr_ze} with \cite[Corollary~4.14.3]{Walters}).
\end{rem}

\begin{rem}\label{Rem:SOEtoNonRigid}
    Every Cantor minimal system $(X,T)$ has in its strong orbit equivalence class a Cantor minimal system which is non-partially measure-theoretically rigid with respect to at least one of its ergodic invariant probability measures. It follows from Theorem \ref{thm:pmtr_ze} and the result by Sugisaki \cite{Sugisaki2007} which states that for every $\alpha \in [1,\infty)$ there is a minimal subshift of entropy $\log \alpha$ in the strong orbit equivalence class of $(X,T)$ (see also \cite{Durand2010}).
\end{rem}

\subsubsection{Topological Rigidity}

\begin{dfn}\label{definition:u-rigid}
	Let $T\colon X\to X$ be a dynamical system. We say that $T$ is \emph{topologically rigid} if for every $\eps>0$ there exists an $k\in\N$ such that $d_{\text{sup}}(T^{k},\id)<\eps$.
\end{dfn}

In some literature topologically rigid is called uniformly rigid. Clearly, every topologically rigid system is invertible.

\begin{lem}\label{lem:isometry}
 An isometry on a compact metric space is topologically rigid on each of its transitive components.
\end{lem}

\begin{proof}
Clearly isometries are invertible. Let $x$ be arbitrary. By compactness of $X$, the omega-limit set $\omega(x) \neq \emptyset$. If $x \notin \omega(x)$, then $\delta := d(x,\omega(x)) > 0$. Take $n \geq 1$ such that $d(T^n(x),y) < \delta$. Then because $\omega(x)$ is backward invariant, $\delta > d(x,T^{-n}(y) \geq d(x, \omega(y))$, a contradiction.
%
% As $X$ is compact, $\overline{\orb(x)}$ is compact, and since isometries are continuous,
% there is a $T$-invariant probability measure supported on $\overline{\orb(x)}$.
%By Poincar\'e recurrence,
Hence, there is a sequence $n_k \to\infty$ such that $T^{n_k}(x) \to x$.
 Let $z \in \overline{\orb(x)}$, so there is $m_k \to\infty$ such that $T^{m_k}(x) \to z$.
 Then
 \begin{eqnarray*}
  d(T^{n_k}(z),z) &\leq& d(T^{n_k}(z),T^{n_k+m_k}(x),z)  + d(T^{n_k+m_k}(x),T^{m_k}(x))
  + d(T^{m_k}(x),z)  \\
  &=& d(z,T^{m_k}(x))  + d(T^{n_k}(x),x) + d(T^{m_k}(x),z) \to 0
 \end{eqnarray*}
 as $k \to\infty$. Hence $(n_k)_{k \geq 1}$ is a topological rigidity sequence for $\overline{\orb(x)}$.
\end{proof}

It follows immediately that transitive isometries (such as irrational rotations and odometers) are topologically rigid, but we cannot really weaken transitivity, because the twist map in Remark~\ref{Thm:rigidity_any_inv_mu_a_n_b_n} is rigid on each transitive part, but Lebesgue measure, as global invariant measure, is not even partially rigid.

Recall that a measure on a topological space is called regular if every measurable set can be approximated from above by open measurable sets and from below by compact measurable sets. In particular, any Borel probability measure on a compact metric space
is regular.

\begin{prop}
If $T$ is topologically rigid, then $(X,T, \mu)$ is measure-theoretically rigid for every $T$-invariant measure $\mu$.
\end{prop}
\begin{proof}
Since $T$ is topologically rigid, it is invertible.
Fix any Borel set $A$ and any invariant measure $\mu$. Let $t_n$ be a sequence provided by topologically rigid, i.e., $\lim_{n\to\infty}d_{\text{sup}}(f^{t_n},\id)=0$.
Fix any $\varepsilon>0$.
Since $\mu$ is regular, there exists a closed set $C\subset A$ and an open set $C\subset U$ such that $\mu(A\setminus C)<\varepsilon/4$ and $\mu(U\setminus C)<\varepsilon/4$.
There is $N>0$ such that $T^{t_n}(C)\subset U$ and $T^{-t_n}(C)\subset U$ for every $n>N$.
But then
\begin{eqnarray*}
\mu(T^{t_n}(A) \Delta A)&\leq& \mu(T^{t_n}(C) \Delta C)+\mu(T^{t_n}(A\setminus C))+\mu(A\setminus C)\\
&\leq& 2\mu(U\setminus C)+2\mu(A\setminus C)<\varepsilon
\end{eqnarray*}
for every $n>N$. This completes the proof.
\end{proof}

\begin{dfn}\label{def:regrec}
A point $x$ of a dynamical system $(X,T)$ is \textit{regularly recurrent} if for every open neighborhood $U$ of $x$, the set $\{n \in \N : T^n(x) \in U\}$ contains an infinite arithmetic progression.
\end{dfn}

The above theorem implies the following rather simple observation:

\begin{thm}\label{thm:URtranistive}
If $(X,T)$ is a transitive and topologically rigid dynamical system and $X$ is totally disconnected, then $(X,T)$ is conjugate with an odometer (possibly a trivial one).
\end{thm}

\begin{proof}
If $X$ has an isolated point $x$ then it $X$ is finite and $T$ permutes its points, i.e., $X$ is a single periodic orbit. Hence assume on the contrary that $X$ does not have isolated points, hence it is a Cantor set.

Fix any clopen partition of $X$, say $U_1,\ldots,U_k$ and let $\eps>0$ be such that $\dist (U_i,U_j)>2\eps$ for all $i\neq j$. By topological rigidity, there is $m$ such that $d(T^m(x),x)<\eps$ for all $x\in X$. This in particular implies that $T^m(U_i)\subset U_i$ for any $i$, and therefore, for every $x\in U_i$ we have $m\N\subset  \{n \in \N : T^n(x) \in U_i\}$. By the above we easily obtain that every point in $x$ is regularly recurrent. But transitivity implies that $X$ is an orbit closure of one of these points. The proof is completed by Theorem~\ref{thm:odometers}.
\end{proof}

\begin{cor}\label{cor:decomposition}
If $(X,T)$ is a topologically rigid Cantor dynamical system, then $X$ is a disjoint union of minimal systems, each conjugate to an odometer (possibly a trivial one).
\end{cor}
\begin{proof}
By topological rigidity, every point $x \in X$ is recurrent. Therefore its $\omega$-limit set defines a transitive dynamical system, so the conjugacy is obtained by Theorem~\ref{thm:URtranistive}.
\end{proof}

It is not hard to see that the bases of odometers in the decomposition provided Corollary~\ref{cor:decomposition} cannot be selected uniformly. To see this, fix a sequence of circles $S_n=\{z : |z|=r_n\}$ in the plane with diameters $r_n$ converging to $0$. Define $T$ and $X=\sup_n C_n \cup \{0\}$ in the following way. On each $S_n$ the map $T$ is a rotation by angle $1/n$ and $C_n\Subset S_n$ is an arbitrarily chosen Cantor set consisting of points of period $n$. Then it is clear that $X$ is a Cantor set and $T^{n!}|_{C_i}=\id_{C_i}$ for $i=1,2,\ldots,n$ showing that $T$ is topologically rigid.

\subsection{Rigidity and (weak) mixing}

Mixing precludes partial rigidity, because by definition,
for any $\alpha > 0$ and measurable set $A$ with $\mu(A) < \alpha$,
we have $\lim_{n \to\infty} \mu(A \cap T^n(A)) = \mu(A)^2 < \alpha \mu(A)$.
However, rigidity is compatible with weak mixing, because then the above convergence only has to happen along sequences
$(n_k)$ of full density. The complement $\N \setminus \{ n_k : k \in \N\}$ could contain the rigidity times.
Rigidity together with weak mixing is the typical situation for interval exchange transformations (IETs), as shown by Ferenczi and Hubert \cite{Fh19}.

It is more surprising that also topologically rigid maps can be weakly mixing.
In 1977, Fathi and Herman 
\cite{FH77} used fast approximation method of Anosov and Katok \cite{AK70},
to study particular diffeomorphisms of manifolds. When restricted to the torus, the residual set in the closure $\overline{\mathcal{O}(\mathbb{T}^2)}$
of maps that arise as approximations
$$
\mathcal{O}(\mathbb{T}^2) = 
\{h\circ R_\alpha 
\circ h^{-1} : h 
\in 
\text{Diff}^\infty(\mathbb{T}^2), \alpha
\in \mathbb{T}^2\} 
$$
consists of weakly mixing and minimal systems (see \cite{KK09} and references therein). It is also obvious that the set of topologically rigid transformations is residual in $\overline{\mathcal{O}(\mathbb{T}^2)}$. This provides examples of minimal, topologically rigid and weakly mixing diffeomorphisms of the torus. Approach from \cite{FH77} was also motivation for Glasner and Maon 
\cite{GM} who used Baire category theorem (together with earlier results of Glasner and Weiss \cite{GW}) to obtain (among other examples) skew-product homeomorphism on torus of any dimension $\mathbb{T}^n$, $n\geq 2$ which is minimal, topologically rigid and weakly mixing. Such examples are not possible on the circle, however in \cite{BCO} the authors observed that example of Handel \cite{Handel} obtained by fast approximation technique provides 
minimal, topologically rigid and weakly mixing homeomorphism of the pseudo-circle (a `pathological' one-dimensional continuum). Summarizing all the above examples we obtain the following:

\begin{rem}
There are examples of a connected space $X$ of any given topological dimension $n=1,2,\ldots$ admitting minimal, topologically rigid and weakly mixing dynamical systems.
%(e.g. such examples exists on torus $\mathbb{T}^n$ for any $n$). {\red Put reference later} Such a system has only trivial equicontinuous factors.
\end{rem}

\subsection{Basic definitions and facts about Bratteli diagrams}

In this subsection, we present basic definitions and results about Bratteli diagrams that we will use throughout the paper.
For more information on Bratteli diagrams see e.g.\ surveys \cite{Durand2010}, \cite{BezuglyiKarpel2016},\cite{BezuglyiKarpel_2020} and \cite[Section 5.4]{Bruin_book}.

\begin{dfn}
    A {\it Bratteli diagram} is an infinite graph $B=(V,E)$ divided into disjoint union of vertex sets $V =\bigsqcup_{i\geq 0}V_i$ and  edge sets $E=\bigsqcup_{i\geq 1}E_i$, where

(i) $V_0=\{v_0\}$ is a single point;

(ii) $V_i$ and $E_i$ are finite sets for all $i$;

(iii) there exists a source map $s \colon E \rightarrow V$ and a target map (or range map) $t \colon E \rightarrow V$ such that $s(E_i)= V_{i-1}$ and $t(E_i)= V_i$ for all $i \geq 1$.
\end{dfn}

We will call the set of vertices $V_i$ the \textit{$i$-th level} of the diagram $B$.
Let 
$$
X_B = \{x = (x_i)_{i = 1}^{\infty} : x_i \in E_i \text{ and } t(x_i) = s(x_{i+1}) \text{ for } i \geq 1\}
$$ 
be the set of all infinite paths in $B$ that start from $v_0$. The set $X_B$ is endowed with the topology generated by cylinder sets
$[\overline{e}]$, where $\overline{e} = (e_1, ... , e_n)$,  $n \in \mathbb N$ is a finite path and
$[\ov e]:=\{x\in X_B : x_i=e_i,\; i = 1, \ldots, n\}$. With this topology,
$X_B$ is a  0-dimensional compact metric space. 
This topology is generated by the 
following metric: for $x = (x_i), \, y = (y_i) \in X_B$, set 
$$
d(x, y) = \frac{1}{2^N},\  \mbox{ where } N = \min\{i \in \N :  x_i \neq y_i\}.
$$
For vertices $v,w \in V$, let $E(v,w)$ denote the set of finite paths between $v$ and $w$. %Denote by symbol $|\cdot |$ the cardinality of a set.

To define a dynamical system on the path space of a Bratteli diagram, we need to take a linear order $>$ on each set $t^{-1}(v)$ for $v
\in V\setminus V_0$. This order defines a partial order on 
the sets of edges $E_i$ for $i=1,2,\ldots$, edges $e,e'$ are comparable if and only if $t(e)=t(e')$. A Bratteli diagram $B=(V,E)$ together with a partial order $>$ on $E$ is called 
\textit{an ordered Bratteli diagram} $B=(V,E,>)$. We call a (finite or infinite) path $e= (e_i)$ 
\textit{maximal (respectively minimal)} if every
$e_i$ has a maximal (respectively minimal) number among all 
elements from $t^{-1}(t(e_i))$. Denote by $X_{\max}$ and $X_{\min}$ the sets of all infinite maximal and all infinite minimal paths in $X_B$.

\begin{dfn}\label{Def:VershikMap} Let $B=(V,E,>)$ be an ordered Bratteli diagram.  Given $x = (x_i)_{i = 1}^{\infty}\in X_B\setminus X_{\max}$, 
let $m$ be the smallest number such that $x_m$ is not maximal. Let 
$y_m$ be the successor of $x_m$ in the set $t^{-1}(t(x_m))$.
Set $\varphi_B(x)= (y_1, y_2,...,y_{m-1},y_m,x_{m+1},...)$
where $(y_0, y_1,..., y_{m-1})$ is the unique minimal path in $E(v_0, s(y_{m}))$. If 
$\varphi_B$ admits an extension 
to the 
entire path space $X_B$ such that $\varphi_B$ becomes a homeomorphism of $X_B$, then $\varphi_B$ is called a \textit{Vershik map}, and 
the system $(X_B,\varphi_B)$ is called a \textit{Bratteli-Vershik system}.

\end{dfn}

If $|X_{\min}| = |X_{\max}| = 1$, 
then the Vershik map can be extended to a homeomorphism of $X_B$ by sending a unique maximal path into the unique minimal path. In the case when $|X_{\max}| > |X_{\min}| = 1$, then the Vershik map can be extended to a continuous surjection of $X_B$ by mapping all maximal paths into the unique minimal path.

\begin{dfn}
Let $x =(x_n)$ and $y =(y_n)$  be two paths from $X_B$. We
 say that $x$ and $y$ are {\em tail equivalent} (in symbols,  $(x,y)
  \in   \mathcal R$)  if there exists some $n$ such that $x_i = y_i$ for
   all  $i\geq n$.
\end{dfn}

Let $\mu$ be a Borel probability non-atomic measure. We say that $\mu$ is
   $\mathcal R$-invariant measure
 %means that  
 if $\mu([\ov e]) = \mu([\ov e'])$ for any two  finite paths
 $\ov e,\ov e' \in  E(v_0, v)$, where $v \in  V_n$ is an arbitrary vertex, and
 $n \geq1$. Denote by $\mathcal{M}_1(\mathcal{R})$ the set of all Borel probability $\mathcal{R}$-invariant measures and by $\mathcal{M}_1(\varphi_B)$ the set of all Borel probability $\varphi_B$-invariant measures.

A homeomorphism without periodic points is called aperiodic if it has no periodic points.
We say that equivalence relation $\mathcal{R}$ is aperiodic if for every $x \in X_B$ its equivalence class is infinite. 

\begin{lem}\cite{BezuglyiKwiatkowskiMedynetsSolomyak2010}
    Let $B=(V,E,>)$ be an ordered Bratteli diagram which admits an aperiodic Vershik map $\varphi_B$ and let the tail equivalence relation $\mathcal{R}$ be aperiodic. Then $\mathcal{M}_1(\mathcal{R}) = \mathcal{M}_1(\varphi_B)$.
\end{lem} 

In this paper, we will consider only such Bratteli 
diagrams for which the tail equivalence relation and the Vershik map are aperiodic. 

Let $X_n(v)$ denote the set of all paths from $X_B$ that go through the vertex $v \in V_n$, and $h_n(v)$ denote the cardinality of the set of all finite paths  between $v_0$ and $v$. We call $X_n{(v)}$ the \textit{tower} corresponding to the vertex $v$ on level $n$, and $h_n(v)$ the \textit{height} of the tower $X_n(v)$.

\begin{dfn}
Given a Bratteli diagram $B$, the {\em $n$-th  incidence matrix}
$F_{n}=(f^{(n)}_{v,w}),\ n\geq 1,$ is a $|V_{n}|\times |V_{n-1}|$
matrix such that 
$$
f^{(n)}_{v,w} = |\{e\in E_{n} : t(e) = v, s(e) = w\}|
\text{ for } v\in V_{n} \text{ and } w\in V_{n-1}. 
$$
A Bratteli diagram is called \textit{stationary} if $F_n = F$ for every $n \geq 2$. Then the matrix $F$ is called the incidence matrix of the diagram. Unless stated otherwise, we will assume that every diagram has a \emph{``simple hat''}, which means that there is a single edge from the vertex $v_0$ to each vertex $v \in V_1$.
\end{dfn}

\begin{dfn} \label{telescoping_definition}
Let $B$ be a Bratteli diagram, and $n_1 = 0 <n_2<n_3 < \ldots$ be a 
strictly increasing sequence of integers. The {\em telescoping of $B$ with respect to $
\{n_k\}_{k = 1}^{\infty}$} is the Bratteli diagram $B'$, whose $k$-level vertex set $V_k'$ is 
$V_{n_k}$ and whose incidence matrices $\{F_k'\}_{k = 1}^{\infty}$ are defined by
\[F_k'= F_{n_{k+1}-1} \; \cdots \; F_{n_k},\]
where $\{F_n\}_{n = 1}^{\infty}$ are the incidence matrices for $B$.
\end{dfn}

\begin{dfn}
    We say that a Bratteli diagram $B = (V,E)$ has \textit{finite rank} if there exists some
 $k\in \mathbb N$ such that $|V_n| \leq k$  for all $n\geq 1$. For a finite rank diagram $B$, we  say that $B$ has
\textit{rank $d$} if  $d$ is the smallest integer such that $|V_n|=d$
 for infinitely many $n$. A Cantor dynamical system has \textit{topological finite rank} $d$ if it is topologically conjugate to a Vershik map acting on a Bratteli diagram of rank $d$ and $d$ is the smallest such number.  
\end{dfn}

\begin{rem} If there is an increasing sequence of natural numbers $\{n_k\}_{k = 1}^{\infty}$ such that $|V_{n_k}| = d$ for all $k \in \mathbb{N}$ then after telescoping with respect to  $\{n_k\}_{k = 1}^{\infty}$ we obtain a Bratteli diagram of rank $d$ which has exactly $d$ vertices on each level.
\end{rem}

A \textit{$(p_n)$-odometer} is a Bratteli diagram with $V_n = \{ v \}$ and $p_n$ edges in $E_n$ for all $n\in \N$. This gives a Vershik homeomorphism for any ordering of incoming edges. For more information and other definitions of odometers as adding machines or inverse limits, see~\cite{Downar}. It was proved in \cite{DownarowiczMaass2008}, that every Cantor minimal system of finite topological rank is either an odometer or a subshift:

\begin{thm}\cite{DownarowiczMaass2008}\label{Thm:DownarMaass}
    Every Cantor minimal system of finite topological rank $d > 1$ is expansive.
\end{thm}

The following definition can be found for instance in \cite{BezuglyiKwiatkowskiMedynetsSolomyak2013}:
\begin{dfn}
    A Bratteli diagram of finite rank is of \textit{exact finite rank} if there is a finite invariant measure $\mu$ and a constant $\delta > 0$ such that after telescoping $\mu(X_n(v)) \geq \delta$ for all $n \in \mathbb{N}$ and $v \in V_n$.
\end{dfn}

\begin{thm}\cite{BezuglyiKwiatkowskiMedynetsSolomyak2013}\label{thm:BKMS13ExactFinRank}
Let $B = B(F_n)$ be a Bratteli diagram of finite rank. Then
\begin{enumerate}
    \item if there is a constant $c > 0$ such that $\frac{m_n}{M_n} \geq c$ for all $n$, where $m_n$ and $M_n$ are the smallest and the largest entry of $F_n$ respectively, then $B$ 
    is of exact finite rank,

    \item if $B$ is of exact finite rank, then $B$ is uniquely ergodic.
\end{enumerate}
\end{thm}

The following theorem can be found for instance in \cite{Danilenko2016}:
\begin{thm}\label{thm:Danilenko}
  Let $\varphi_B$ be a Vershik map on a Bratteli diagram of exact finite rank. Then $\varphi_B$ is partially rigid with respect to the unique ergodic invariant probability measure $\mu$ on $B$.
\end{thm}

%\subsection{Toeplitz systems}

\medskip
\section{Rigidity for Toeplitz systems}\label{Sect:Toeplitz}    
 
 A topological factor map $\pi: X \mapsto Y$ between two dynamical systems $(X,T)$ and $(Y,S)$ is a continuous surjective map such that $\pi \circ T = S\circ\pi$. We say it provides an \textit{almost $1$-$1$ extension} if the set of points in $Y$ having singleton fibers is residual. For minimal systems $(Y,S)$ it suffices to show one such singleton fiber exists.

The following is an old characterization of regularly recurrent points (e.g. see \cite[Theorem~5.1]{Downar}):

\begin{thm}\label{thm:odometers}
    A topological dynamical system $(X,T)$ is a minimal almost 1-1 extension of an odometer $(G_s,\tau)$ if and only if it is the orbit closure of a regularly recurrent point. The set of all regularly recurrent points in $(X,T)$ coincides with the collection of all singleton fibers, and it is a dense $G_\delta$ subset of $X$.
\end{thm}

Let $\omega$ be a regularly recurrent sequence in $\mathcal{A}^{\mathbb{Z}}$ under the left shift $\sigma$ (for a finite alphabet $\mathcal{A}$). Let $X$ be the orbit closure of $\omega$ under $\sigma$. If $\omega$ is not periodic, we call it a\textit{ Toeplitz sequence} and $(X,\sigma)$ the \textit{Toeplitz system} generated by $\omega$.

For $p\in \N$ the \textit{$p$-periodic part} of $\omega$ is $Per_p(\omega)$ the set of all positions $i\in\Z$ such that
$$
    \omega_i = \omega_{i+np} \text{ for all $n\in\Z$.}
$$
We construct $S_p(\omega) \in (\mathcal{A}\cup \{ * \})^\mathbb{Z}$, the \textit{$p$-skeleton} of $\omega$, by replacing all $\omega_i$ in $\omega$ by $*$ for $i\not\in Per_p(\omega)$.

Then $\omega$ has \textit{periodic structure} $p=(p_i)_i$ if $p$ is an increasing sequence of integers with $p_0>1$ such that
\begin{itemize}
    \item for every $i$ the $p_i$-skeleton of $\omega$ is not periodic with any period smaller then $p_i$,
    \item $p_i \ | \ p_{i+1}$ for all $i$ and
    \item for all $n \in \N$, there exists $i$ with $n\in Per_{p_i}(\omega)$.  
\end{itemize}
For more information, see \cite{Kurka2003}, \cite{Williams84}.

\begin{dfn}
    Let $\omega$ be a Toeplitz sequence with periodic structure $p$. Set
    $$
    q_i = |S_{p_i}(\omega)|_*
    $$
    as the number of occurrences of $*$ in $S_{p_i}(\omega)$.
    The sequence $\omega$ is regular Toeplitz if
    $$
    \delta(\omega) = \lim_{i \rightarrow \infty}\frac{q_i}{p_i} = 0.
    $$
\end{dfn}

\begin{thm}[\cite{JacobsKeane1969} ]\label{Thm:regularToeplitzProperties}
    If $\omega$ is a regular Toeplitz sequence, then $\overline{\mathcal{O}(\omega)}$ has zero topological entropy and is strictly ergodic. 
\end{thm}

%{\color{red} [Not done yet]}
%---------------- to be developed
%Theorem 14.2. Under the assumption (SAR), the Toeplitz flow (X,σ) is Borel* conjugate to the skew product (Gs × Y, T ), where
%1⁄2(j+1,σ(y)), wheneverj∈D T(j,y)= (j+1,y), wheneverj∈/D.
%(the conjugating map is defined between the full sets E and E × Y , respectively).

%By results of S. Williams \cite{Wil}, for any preset symbolic system $Y$ there exists an Oxtoby sequence satisfying (SAR) and with associated skew product defined by $Y$ (see also Wiliams' construction presented in \cite{Downar}).

Let $D$ be a subset of a dynamical system $(X,T,\mu)$ such that $\mu(D) > 0$. By $T_D$ we will denote the first return time map induced by $T$ on $D$:
$$
T_D(x) = T^{n(x)}(x),\qquad \text{ for }
n(x) = \min\{n \in \N : T^n(x) \in D\}.
$$

The Poincar\'e recurrence theorem ensures that $T_D$ is defined $\mu$-almost everywhere on $D$, and it preserves conditional measure 
$\mu_D$ on $D$ given by
$\mu_D(B)=\mu(B)/\mu(D)$ for all Borel sets $B\subset D$. In what follows, we will associate with a Toeplitz system a special set $\underline{D}$ and return map $\sigma_{\underline{D}}$. The aim of this construction is showing that there are Toeplitz systems with zero entropy which are still not measure-theoretically rigid.

Let $\mathbf{s} = (s_m)_{m \in \N}$ be a sequence of positive integers such that $s_m$ divides $s_{m+1}$. Denote the corresponding odometer by
$$
G_\mathbf{s} = \invlim_m \mathbb{Z}_{s_m},
$$
where the bonding maps are $f_m \; \colon \; \mathbb{Z}_{s_{m+1}} \rightarrow \mathbb{Z}_{s_m}$ with $f_m(x_{m+1}) = x_{m+1}\pmod{s_m}$.
Denote by $\tau$ the translation by unity on $G_\mathbf{s}$ (i.e., addition of $1$ at each coordinate). Let $(G_\mathbf{s},\tau,\lambda)$ be the associated measure-theoretical dynamical system, where $\lambda$ is the Haar measure on $G_\mathbf{s}$.

Fix a Toeplitz sequence $\omega$ and let $(X_\omega,\sigma,\mu)$ be the associated Toeplitz system. Then $(X_\omega,\sigma)$ is an almost 1-1 extension of $(G_\mathbf{s},\tau)$ which is its maximal equicontinuous factor (and $\mathbf{s}$ depends on $X_\omega$, in particular on $\omega$). Let $\pi \colon X_\omega\to G_\mathbf{s}$ be the associated factor map. For an odometer we have a natural map $n\mapsto \tau^n(\mathbf{1})$ where $\mathbf{1}=(1,1,1,\ldots)$. Then the topology on $G_\mathbf{s}$ induces by this map a natural topology on $\mathbb{Z}$ (similarly for $\N$). Any function $f\colon \mathbb{Z}\to \mathcal{A}$, where $\mathcal{A}$ is a finite set (an alphabet) with discrete topology, continuous with respect to this topology is called a semi-cocycle on $G_{\mathbf{s}}$. We can view $\omega: n\mapsto \omega(n)$ as a semi-cocycle (see \cite[Theorem~7.1]{Downar}). Denote by $F$
the graph of semi-cocycle $\omega$ in $G_\mathbf{s}\times \mathcal{A}$, where $\Z$ is identified with the two-sided orbit of $\mathbf{1}$ in $G_\mathbf{s}$. Denote by $D\subset G_\mathbf{s}$ the set of $\mathbf{j}\in G_\mathbf{s}$ such that the section
$\{a\in \mathcal{A} : (\mathbf{j},a)\in F\}$ is not a singleton. Note at this point that each section has always at least one point, so we can view $D$ as the set of points at which $\omega$ cannot be extended to a continuous map on $G_\mathbf{s}$. If $\lambda(D_\omega)=0$ then $(X,\sigma,\mu)$ is measure-theoretically isomorphic to its maximal equicontinuous factor $(G_\mathbf{s},\tau,\lambda)$. This is exactly the case when $\omega$ is a regular Toeplitz sequence \cite{Markley1975}. 

\begin{rem}\label{Rem:RegToeplitz}
If a Toeplitz system $(X_\omega,\sigma)$ is defined by a regular Toeplitz sequence, then its unique invariant measure $\mu$ is measure-theoretically rigid.
This does not extend onto all Toeplitz systems, since there are Toeplitz systems with positive entropy, and these are not even partially measure-theoretically rigid,
see Theorem~\ref{thm:pmtr_ze}.
\end{rem}

In what follows, we will assume that $\omega$ is an irregular Toeplitz system, i.e., $\lambda(D_\omega) > 0$.
Fix any $x\in X_\omega$ and let $\mathbf{j}=\pi(x)$. We define aperiodic part of $x$, denoted $\text{Aper}(x)\subset \Z$ by putting $n\in \text{Aper}(x)$ if and only if $\tau^n(\mathbf{j})\in D_\omega$.
Since $\omega$ is irregular, there is a set $E$ such that $\lambda(E)=1$ and if we denote $\underline{E}=\pi^{-1}(E)$ then $\text{Aper}(x)$ does not have upper nor lower bound for every $x\in \underline{E}$.
For each $\mathbf{j}\in E$ we define its aperiodic readouts as
$$
Y_\mathbf{j}=\{x|_{\text{Aper}(x)} : \pi(x)=\mathbf{j}\}.
$$
If the sets $Y_\mathbf{j}$ are the same over all $\mathbf{j}\in E$ then we say that $X$ satisfies condition (SAR) (Same Aperiodic Readouts). If we want to emphasize the unique set $Y=Y_\mathbf{j}$ then we say that (SAR) is satisfied with readouts equal to $Y$. The Williams' classical construction of a Toeplitz sequence satisfying (SAR) and with $Y$ equal to an arbitrarily preset subshift is presented in \cite[Sec.~14]{Downar}.

Let $(X, T)$ and $(Y, S)$  be dynamical systems and let $M_T(X), M_S(Y)$ be the spaces of invariant Borel probability measures for these systems, respectively.
A map $\rho \colon X \to Y$ is called a \textit{Borel* conjugacy} if:
\begin{enumerate}
\item $\rho$ is a Borel-measurable bijection,
\item $\rho \circ T= S \circ \rho$,
\item the map acting on measures (also denoted by $\rho$) defined by $\rho (\mu)(B) = \mu(\rho^{-1}(B))$ is a homeomorphism with respect to the weak* topology of measures.
\end{enumerate}
The above preparation is the main step to apply the following tool:% \cite[Theorem 14.2]{Downar}.

\begin{thm}\label{thm:sar:toplitz}\cite{Downar}
Fix any subshift $Y$.
Let $\omega$ be a Toeplitz sequence satisfying
(SAR) with readouts equal to $Y$, let $(X_\omega,\sigma)$ be the associated Toeplitz system and let $(G_\mathbf{s},\tau,\lambda)$ be its maximal equicontinuous factor. Denote by $D_\omega\subset G_\mathbf{s}$ the set of discontinuities of semi-cocycle $\omega$.
Then the Toeplitz system $(X_{\omega},\sigma)$ is Borel* conjugate to the skew product $(G_\mathbf{s} \times Y, T)$, where
\begin{equation}\label{eq:skewT}
T(\mathbf{j},y)=\begin{cases}
(\mathbf{j}+\mathbf{1},\sigma(y))& \text{if }\mathbf{j}\in D_\omega,\\
(\mathbf{j}+\mathbf{1},y)& \text{if }\mathbf{j}\not\in D_\omega.\\
\end{cases}
\end{equation}
\end{thm}

Now we can prove the following theorem:

\begin{thm}\label{thm:skew}
Fix any measure-theoretically strong mixing subshift $(Y,\sigma,\nu)$ and let $D\subset G_\mathbf{s}$ be a set such that $\lambda(D)>0$,
where $\lambda$ is the Haar measure.
Then the skew product $(G_\mathbf{s} \times Y, T,\lambda\times \nu)$, where
$$
T(\mathbf{j},y)=\begin{cases}
(\mathbf{j}+\mathbf{1},\sigma(y))& \text{if }\mathbf{j}\in D,\\
(\mathbf{j}+\mathbf{1},y)& \text{if }\mathbf{j}\not\in D\\
\end{cases}
$$
is not partially measure-theoretically rigid.
\end{thm}
\begin{proof}
Fix any $\alpha\in (0,1]$ and
fix any cylinder set $A\subset Y$ such that $\nu(A)<\alpha/4$. Fix any $\eps<\nu(A)^2$ and let $N_\eps>0$ be such that $\nu(\sigma^n(A)\cap A)<\nu(A)^2+\eps$ for each $n\geq N_\eps$.

For each $x\in X$ and $m \in \N$, let
$r_m(x)=|\{0\leq j\leq m : \tau^j(x)\in D\}|$
be the number of visits of $x$ in $D$ in $m$ iterations.
Denote $D^m_i=\{x\in D : r_m(x)=i\}$ and note that
$D=\bigcup_{i=1}^m D^m_i.$
Take $M$ so large that
$$
\lambda\left(\bigcup_{i=1}^{N_\eps-1}D^m_i\right)<\eps{ \lambda(D)}
$$
for every $m\geq M$.
Note that for each $0\leq i\leq m$ we have
$$
T^m(D^m_i\times A)=\tau^m(D^m_i)\times \sigma^i(A).
$$
But then
\begin{eqnarray*}
\lambda\times \nu(T^m(D\times A)\cap (D\times A))&\leq&
\sum_{i=1}^{m}\lambda(\tau^m(D^m_i)\cap D)\nu (\sigma^i(A)\cap A)\\
&\leq& \sum_{i=1}^{N_\eps-1}\lambda(D_i^m)
+\sum_{i=N_\eps}^m \lambda(D_i^m)(\nu(A)^2+\eps)\\
&\leq & \lambda(D)(\nu(A)^2+2\eps)\leq 
3\lambda(D)\nu(A)^2
<\alpha (\lambda\times \nu) (D\times A).
\end{eqnarray*}
This implies that indeed $T$ is not partially measure-theoretically rigid.
\end{proof}

In this proof, the Toeplitz shift appears as Kakutani tower over a mixing base map.
Such a system would be mixing, if the return times had greatest common divisor $1$, which in this case is not true. Therefore, the mixing of the Kakutani tower fails, but the mixing of the base is still powerful enough to prevent partial rigidity.

\begin{cor}\label{cor:toeplitz_non_part_rigidity}
There is a Toeplitz system $(X,\sigma)$ with zero entropy which is not partially measure-theoretically rigid with respect to any of its invariant measures.
\end{cor}
\begin{proof}
Fix any strong mixing measure $\mu$ of zero entropy. By the Jewett-Krieger theorem there exists a uniquely ergodic subshift $(Y,\sigma)$ with its unique measure isomorphic to $\mu$.
Williams' Construction
provides a Toeplitz sequence satisfying (SAR)
and with aperiodic readouts equal to $Y$. Then by Theorem~\ref{thm:sar:toplitz}
we obtain a Toeplitz system $(X,\sigma)$ which is Borel* conjugate to the skew product $(G_\mathbf{s} \times Y, T)$ given by \eqref{eq:skewT}. By Theorem~\ref{thm:skew} we obtain that $(G_\mathbf{s} \times Y, T,\lambda\times \nu)$
is not partially measure-theoretically rigid.

Next, let $\mu$ be any other invariant measure of $T$. Then induced measure $\mu_{D\times Y}$ %\mu_D$ 
is invariant for
$T_D=\tau_D\times \sigma$. But Williams' construction ensures that $\tau_D$ is conjugate to an odometer $G_{\mathbf{s}'}$ on some scale $\mathbf{s}'$. Since $(Y,\sigma,\nu)$ is strongly mixing, it is disjoint from $(G_{\mathbf{s}'},\tau,\lambda)$, and therefore $\mu_{D\times Y}=\lambda_D\times \nu$ since it is joining of $\lambda_D$ and $\nu$. But $\mu_{D\times Y}=(\lambda\times \nu)_{D\times Y}$ so we must have $\mu=\lambda\times \nu$.
\end{proof}

\medskip
\section{Rigidity on Bratteli diagrams}\label{Sect:ext_from_odom}

In this section we investigate how the structure of a Bratteli diagram representing a Cantor dynamical system can show rigidity. We use the diagram to control ergodic measures, specifically if the measure is the finite extension of an odometer measure. Further we use this result to study Toeplitz systems by their Bratteli-Vershik representation and show rigidity in examples with different ergodic measures.

\begin{dfn}\label{Def:consec_order}
An order on a Bratteli diagram is called \textit{consecutive} if whenever we have edges $e,f,g$ such that $e \leq f \leq g$ and $s(e) = s(g)$ then $s(f) = s(e) = s(g)$.
We say that a subdiagram $\ov B = (\ov W, \ov E)$ of $B = (V,E,>)$ is \textit{consecutively ordered} if for any $n \geq 0$ and $e\leq g \in \ov E$ with $s(e) = s(g)$, if there exists an edge $f \in E$ with $e\leq f\leq g$ in the ordering of the full diagram $B$, then $s(f) = s(e) = s(g)$.
%, any $v \in W_{n+1}$, and any $w \in W_n$, from $\min\{e : e \in E(v,w)\} \leq f \leq \max\{e : e \in E(v,w)\}$ it follows that $s(f) = w$.
\end{dfn}

Figure~\ref{Fig:ConsecutiveOrdering} (left) illustrates the notion of a consecutive order. For more details, see e.g. \cite{Durand2010}. 
One of the examples of a consecutive order is a left-to-right order, when for every vertex all incoming edges are enumerated from left to right as they appear in the diagram. If $B$ has a consecutive order, then any of its vertex subdiagrams is consecutively ordered. The stationary diagram in Figure~\ref{Fig:ConsecutiveOrdering} (right) is not consecutively ordered, but for instance the subdiagram which is a vertical odometer passing through the second vertex on each level, is consecutively ordered.  
%new notion

\begin{figure}
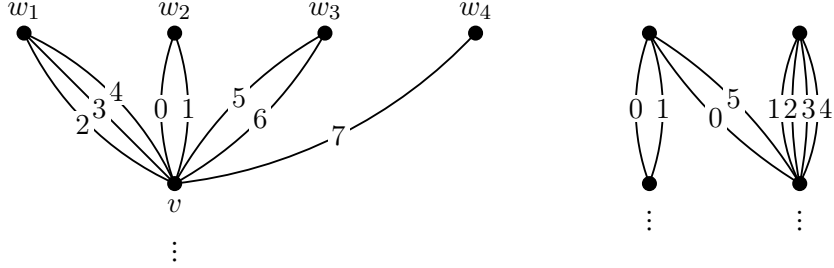

\begin{minipage}[l]{0.45\textwidth}
\unitlength=1cm
\begin{graph}(8,4)
% \graphnodesize{0.2}
% \roundnode{V0}(3,6)
%  %\nodetext{V0}(-1,0){$V_0$}
  % The first level vertices
 \roundnode{V10}(1,3) 
 \roundnode{V11}(3,3)
 \roundnode{V12}(5,3)
 \roundnode{V13}(7,3)
 \nodetext{V10}(0,0.3){$w_1$}
 \nodetext{V11}(0,0.3){$w_2$}
 \nodetext{V12}(0,0.3){$w_3$}
 \nodetext{V13}(0,0.3){$w_4$}
 
 % The second level vertices
 \roundnode{V21}(3,1)
\nodetext{V21}(0,-0.3){$v$}
  
 %
 % EDGES
 \graphlinewidth{0.025}
% % First level
%  \edge{V0}{V11}
%  \edge{V0}{V12}

 % Second level
 \bow{V21}{V10}{0.09}
  \bow{V21}{V10}{-0.09}
  \bowtext{V21}{V10}{0.11}{2}
  \bowtext{V21}{V10}{-0.11}{4}
  \edge{V21}{V10}
  \edgetext{V21}{V10}{3}
  
 \bow{V21}{V11}{0.09}
  \bow{V21}{V11}{-0.09}
  
  \bowtext{V21}{V11}{0.09}{0}
  \bowtext{V21}{V11}{-0.09}{1}

 \bow{V21}{V12}{0.07}
  \bow{V21}{V12}{-0.07}
  
  \bowtext{V21}{V12}{0.07}{5}
  \bowtext{V21}{V12}{-0.07}{6}

    \bow{V21}{V13}{-0.09}
    \bowtext{V21}{V13}{-0.09}{7}
    %\edgetext{V21}{V13}{7}

\freetext(3,0.1){\vdots}
\end{graph}
\end{minipage}
\quad
\begin{minipage}[l]{0.45\textwidth}
\unitlength=1cm
\begin{graph}(5,4)
% \graphnodesize{0.2}
% \roundnode{V0}(3,6)
%  %\nodetext{V0}(-1,0){$V_0$}
  % The first level vertices
 \roundnode{V11}(2,3)
 \roundnode{V12}(4,3)
 
 % The second level vertices
 \roundnode{V21}(2,1)
 \roundnode{V22}(4,1)
  
 %
 % EDGES
 \graphlinewidth{0.025}
% % First level
%  \edge{V0}{V11}
%  \edge{V0}{V12}

 % Second level
 \bow{V21}{V11}{0.09}
  \bow{V21}{V11}{-0.09}
  \bowtext{V21}{V11}{0.09}{0}
  \bowtext{V21}{V11}{-0.09}{1}
    \bow{V22}{V11}{0.06}
    \bow{V22}{V11}{-0.06}
    
 \bowtext{V22}{V11}{0.06}{0}
 \bowtext{V22}{V11}{-0.06}{5}
     
 \bow{V22}{V12}{0.05}
 \bow{V22}{V12}{-0.05}
  \bow{V22}{V12}{0.12}
 \bow{V22}{V12}{-0.12}

  \bowtext{V22}{V12}{0.17}{1}
 \bowtext{V22}{V12}{-0.17}{4}
   \bowtext{V22}{V12}{0.06}{2}
 \bowtext{V22}{V12}{-0.06}{3}
 
\freetext(2,0.5){$\vdots$}
\freetext(4,0.5){$\vdots$}
%\freetext(3,0.1){.\,.\,.\,.\,.\,.\,.\,.\,.\,.\,.\,.\,.\,.\,.\,.\,.\,.\,.\,.\,.}
\end{graph}
\end{minipage}
\caption{Example of a consecutive order (left) 
and a non-consecutively ordered diagram (right) with consecutively ordered subdiagrams
}\label{Fig:ConsecutiveOrdering}
\end{figure}

\iffalse
\begin{figure}
\unitlength=1cm
\begin{graph}(5,4)
% \graphnodesize{0.2}
% \roundnode{V0}(3,6)
%  %\nodetext{V0}(-1,0){$V_0$}
  % The first level vertices
 \roundnode{V11}(2,3)
 \roundnode{V12}(4,3)
 
 % The second level vertices
 \roundnode{V21}(2,1)
 \roundnode{V22}(4,1)
  
 %
 % EDGES
 \graphlinewidth{0.025}
% % First level
%  \edge{V0}{V11}
%  \edge{V0}{V12}

 % Second level
 \bow{V21}{V11}{0.09}
  \bow{V21}{V11}{-0.09}
  \bowtext{V21}{V11}{0.09}{0}
  \bowtext{V21}{V11}{-0.09}{1}
    \bow{V22}{V11}{0.06}
    \bow{V22}{V11}{-0.06}
    
 \bowtext{V22}{V11}{0.06}{0}
 \bowtext{V22}{V11}{-0.06}{5}
     
 \bow{V22}{V12}{0.05}
 \bow{V22}{V12}{-0.05}
  \bow{V22}{V12}{0.12}
 \bow{V22}{V12}{-0.12}

  \bowtext{V22}{V12}{0.17}{1}
 \bowtext{V22}{V12}{-0.17}{4}
   \bowtext{V22}{V12}{0.06}{2}
 \bowtext{V22}{V12}{-0.06}{3}
 
\freetext(2,0.5){$\vdots$}
\freetext(4,0.5){$\vdots$}
%\freetext(3,0.1){.\,.\,.\,.\,.\,.\,.\,.\,.\,.\,.\,.\,.\,.\,.\,.\,.\,.\,.\,.\,.}
\end{graph}
\caption{Not consecutively ordered diagram with consecutively ordered subdiagrams}\label{Figure1}
\end{figure}
\fi

\begin{rem}
     In general, for a Bratteli diagram of rank bigger than one, the consecutive ordering is not preserved under telescoping.
\end{rem}

A substitution $\theta : A \to (A')^+$ is called \textit{proper} if all words $\theta(a)$ over $a\in A$ start with the same letter and end with the same letter (the starting letter and the ending letter can be different). A substitution $\theta$ is called \textit{left proper} if only the the starting letter in each $\theta(a)$ is the same, and \textit{right proper} if the ending letter is the same. 
Assume that $\theta$ is left proper and $\theta(a) = lu(a)$ for all $a \in A$. Then a substitution $\chi \; \colon A \rightarrow (A')^+$ defined by $\chi(a) = u(a)l$ is \textit {(left) conjugate} of $\theta$ and it is right proper.

If $B$ is an ordered Bratteli diagram, one can read a substitution $\theta_n : V_n \rightarrow V_{n-1}^+$ on $B$ at a level $n \geq 2$ as follows. Let $v \in V_{n}$ and $e_1 < \ldots < e_k$ be the incoming edges to $v$. Then set $\theta_n(v) = w_1 \ldots w_k$, where $w_i = s(e_i)$ for $i = 1, \ldots, k$. For instance, a substitution read from Figure~\ref{Fig:ConsecutiveOrdering} is $\theta(v) = w_2w_2w_1w_1w_1w_1w_3w_3w_4$. Clearly, if diagram $B$ is stationary then $\theta_n = \theta$ for $n \geq 2$.  

%\begin{prop}[\cite{DurandHostSkau1999}]
%Let $B$ be a stationary, properly ordered Bratteli diagram with only single edges between the top vertex and the first level (i.e., it has a {\em simple hat}), and $\theta \; \colon A \rightarrow A^+$the substitution read on $B$.\\
%(i) If $\theta$ is aperiodic, then the system $(X_B, \varphi_B)$ is isomorphic to the system $(X_{\theta}, T_{\theta})$.\\
%(ii) If $\theta$ is periodic, then the system $(X_B, \varphi_B)$ is isomorphic to an odometer with a stationary base.
%\end{prop}

In \cite{DurandLeroy2012}, Bratteli-Vershik representation of $S$-adic shifts is given. 
Let $(A_n)_n$ be a sequence of non-empty finite sets (alphabets) and $A_n^+$ be a set of all finite non-empty words over alphabet $A_n$.
Recall that an \textit{$S$-adic representation} of a subshift $(X, \sigma)$ is a sequence $(\theta_n, a_n)_{n \geq2}$, where $(\theta_n \; \colon A_n \rightarrow A_{n-1}^+)_n$ are substitutions on  $A_n$, $a_n \in A_n$ and $X \subset A_1^{\mathbb{Z}}$ is a set of sequences $x = (x_i)$ such that all words $x_i x_{i+1}\ldots x_j$ appear in some $\theta_2 \theta_3 \ldots \theta_n(a_n)$. 
Denote by $(X_n, \sigma)$ the subshift generated by $(\theta_k, a_k)_{k \geq n}$.

%\tcr{how do we pick $a_n$?}
%\tcr{Is the reformulation below correct?}
\begin{prop}\cite{DurandLeroy2012}\label{Thm:S-adic_Rep_BD}
    Let $(X,T)$ be the minimal $S$-adic subshift defined by a sequence of substitutions $(\theta_n \; \colon A_n \rightarrow A_{n-1}^+,a_n)_n$, where all $\theta_n$ are proper. Suppose that for all $n$ the substitutions $\theta_n$ extend by concatenation to a one-to-one map from $X_n$ to $X_{n-1}$. Then $(X,T)$ is conjugate to $(X_B, \varphi_B)$, where $B$ is the Bratteli diagram such that for all $n \geq 2$ the substitution read on $B$ at level $n$ is $\theta_n$.
\end{prop}

Moreover, the following result holds:
\begin{cor}\cite{DurandLeroy2012}\label{Cor:S-adic_BD}
  Let $(X,T)$ be a minimal $S$-adic subshift defined by $(\theta_n, a_n)_{n \geq 2}$, where $\theta_n$ are left or right proper.  Suppose that for all $n$ the substitutions $\theta_n$ extend by concatenation to one-to-one maps from $X_n$ to $X_{n-1}$. Then $(X,T)$ is conjugate to $(X_B, \varphi_B)$, where $B$ is a Bratteli diagram such that for all $n \geq 2$

  (i) the substitution read on $E_{2n}$ is left proper and equal to $\theta_{2n}$ or its conjugate;

  (ii) the substitution read on $E_{2n+1}$ is right proper and equal to $\theta_{2n+1}$ or its conjugate.
\end{cor}

\subsection{Measure extension from a vertex subdiagram}
Let $\overline W = \{W_n\}_{n>0}$ be a sequence of proper, non-empty subsets $W_n \subset V_n$ and $W_0 = \{v_0\}$. The (vertex) subdiagram $\overline B =  (\ov W, \ov E)$ is a Bratteli diagram defined by the vertices $\ov W = \bigsqcup_{i\geq 0} W_n$ and all the edges $\ov E$ that have both their sources and ranges in $\ov W$. Consider the set $X_{\ov B}$ of all infinite paths whose edges belong to $\ov B$. Let $\widehat X_{\ov B}$ be the subset of paths in $X_B$ that are tail equivalent to paths from $X_{\ov B}$. Let $\ov \mu$ be a probability measure on $X_{\ov B}$ invariant with respect to the tail equivalence relation defined on $\ov B$. Then $\ov \mu$ can be canonically extended to the measure $\widehat {\ov \mu}$ on the space $\widehat X_{\ov B}$ by invariance with respect to $\mathcal R$ (see e.g. \cite{BezuglyiKwiatkowskiMedynetsSolomyak2013}, \cite{BezuglyiKarpelKwiatkowski2015}, \cite{AdamskaBezuglyiKarpelKwiatkowski2017}). 
Specifically, for $w \in W_n$, take a finite path $\ov e \in \ov E(v_0, w)$ which lies in $\ov B$. For any finite path $\ov f \in E(v_0, w)$ from the diagram $B$ with the same range $w$, we set $\widehat{\ov \mu} ([\ov f])  = \ov \mu([\ov e])$. 
To extend $\widehat{\ov \mu}$ to an $\mathcal{R}$-invariant measure on the whole space $X_{B}$, set $\widehat {\ov \mu} (X_B \setminus \widehat{X}_{\ov B}) = 0$. 

Set
$$
\wh X_{\ov B}^{(n)} = \{x = (x_i)\in \wh X_{\ov B} : t(x_i) \in W_i, \ \forall i \geq n\}.
$$
Then $\wh X_{\ov B}^{(n)} \subset \wh X_{\ov B}^{(n+1)}$ and

\begin{equation}\label{Eq:meas_extension}
\widehat{\ov\mu}(\wh X_{\ov B}) = \lim_{n\to\infty} \widehat{\ov\mu}(\wh X_{\ov B}^{(n)}).
\end{equation}

In the case of stationary Bratteli diagrams the unique invariant measure can be computed directly. Let $B$ be a stationary Bratteli diagram with the matrix $A$ transpose to the incidence matrix $F$. A real number $\lambda > 1$ is called a \textit{distinguished eigenvalue} for $A$ if there exists a non-negative vector $\bx$ such that $A\bx = \lambda \bx$. In \cite{BezuglyiKwiatkowskiMedynetsSolomyak2010} it was shown that all ergodic $\mathcal{R}$-invariant probability measures for $X_B$ correspond to distinguished eigenvalues of $A$. Moreover, for every $n \in \mathbb{N}$ and $i \in V_n$, the measure $\mu$ corresponding to the pair $(\bx = (x_i), \lambda)$ takes value 
$$
\mu([\ov e]) = \frac{x_i}{\lambda^{n-1}}
$$
on a cylinder set $[\ov e]$ which ends at a vertex $i \in V_n$.

\subsection{Measures supported on odometers}

\begin{prop}\label{Prop:meas_ext_towers}
    Let $B=(V,E)$ be a Bratteli diagram of arbitrary rank and $\mu$ be a probability invariant measure on $B$ which is an extension of a subdiagram $\ov B = (\ov W,\ov E)$. Then
    \begin{equation}\label{Eq:meas-ext-heights}
    \lim_{n \rightarrow \infty} \sum_{w \in W_n} \mu(X_n{(w)}) = 1.
    \end{equation}
\end{prop}

\begin{proof}
    For every $n \in \N$, we have $\wh X_{\ov B}^{(n)} \subset \bigcup_{w \in W_n} X_n(w)$. By~(\ref{Eq:meas_extension}), we obtain that
    $$
    1 = \mu(X_B) = \mu(\wh X_{\ov B}) = \lim_{n \rightarrow \infty} \mu\left(\wh X^{(n)}_{\ov B}\right) \leq \lim_{n \rightarrow \infty} \mu\left(\bigcup_{w \in W_n} X_n(w)\right).
    $$
    Since the towers of level $n$ are disjoint, we obtain~(\ref{Eq:meas-ext-heights}). 
\end{proof}

The next theorem is the first of two results concerning the rigidity of ergodic invariant measures supported on odometers.

\begin{thm}\label{Thm:rigid_mu_from_odom_withoutorder}
    Let $B$ be an ordered Bratteli diagram with incidence matrices $F_n = (f^{(n)}_{v,w})$ and $\varphi_B$ be the corresponding Vershik map. Let $\mu$ be an ergodic invariant probability measure on $B$. If $\mu$ is an extension from an odometer $\ov B = (\{\ov v\}, \ov E)$ such that
    \begin{equation}\label{eq:qf}
    \frac{\sum_{w \in V_n\setminus \{\ov v\}}f^{(n+1)}_{\ov v, w}}{f^{(n+1)}_{\ov v, \ov v}} \longrightarrow 0 \text{ \ as $n\to \infty$,} 
    \end{equation}
    then the system $(X_B, \varphi_B, \mu)$ is measure-theoretically rigid with rigidity sequence $(h_n(\ov v))_n$.
\end{thm}

\begin{proof}
    Let the finite ergodic measure $\mu$ be the extension from the odometer on vertex $\overline{v}$.
    For simplicity, we will denote by $C_n(w)$ any a cylinder set of level $n$ ending in vertex $w\in V_n$, since all of them carry the same measure $\mu$. Take an arbitrary cylinder set $C_N$ ending in $v\in V_N$ and $\varepsilon>0$.
    Then there exists a level $n > N$ such that 
    $$
    \sum_{w \in V_n\setminus \{\ov v\}} \mu(X_n(w))= \sum_{w \in V_n\setminus \{\ov v\}}(F_n \cdots F_{1})_{w, v_0} \mu(C_n(w))  < \varepsilon
    $$
    and (because \eqref{eq:qf} implies that $f^{(n+1)}_{\ov v, \ov v} \to \infty$)
    $$
    %better to have +1 for further estimates
    \frac{1+\sum_{w \in V_n\setminus \{\ov v\}}f^{(n+1)}_{\ov v, w}}{f^{(n+1)}_{\ov v, \ov v}} < \varepsilon.
    $$
    Decompose $C_N$ into cylinder sets ending in level $n$ and observe that
    $$
    \mu(C_N) = \sum_{w \in V_n} (F_n \cdots F_{N+1})_{w, v} \mu(C_n(w)) \leq (F_n \cdots F_{N+1})_{\ov v, v} \mu(C_n(\ov v)) + \varepsilon.
    $$
    Let $C_n$ be one such subcylinder of $C_N$ ending in $\ov v \in V_n$. Then we can look at the image of $C_n$ under $T^{h_n(\ov v)}$
    \begin{align*}
        \mu(C_n \cap T^{h_n(\ov v)}C_n) &\geq (f^{(n+1)}_{\ov v, \ov v} - 1 - \sum_{w\in V_{n}\setminus\{\ov v\}} f^{(n+1)}_{\ov v, w})\mu(C_{n+1}(\ov v))
    \end{align*}
    as all paths in $C_n$ that use an edge in $E_{n+1}$ that connects $\ov v\in V_n$ to $\ov v \in V_{n+1}$ and that are succeeded by an edge connecting $\ov v\in V_n$ to $\ov v \in V_{n+1}$ in the incoming order to $\ov v \in V_{n+1}$ return to themselves in $C_n$ after $h_n(\ov v)$ steps.
    Thus
    \begin{align}\nonumber
        \mu(C_N \cap T^{h_n(\ov v)}C_N) &\geq (F_n \cdots F_{N+1})_{\ov v, v}\mu(C_n \cap T^{h_n(\ov v)}C_n)\\
         %&\geq (F_n \cdots F_{N+1})_{(\ov v, v)} \left(p_n(\ov v) - \sum_{i\in V_{n+1}\setminus\{\ov v\}} (F_{n+1})_{(i, \ov v)}p_{n+1}(i) - \sum_{j\in V_{n}\setminus\{\ov v\}} (F_{n+1})_{(\ov v, j)}p_{n+1}(\ov v)\right) \\
        %&\geq \mu(C_N) -\varepsilon -\varepsilon - (F_n \cdots F_{N+1})_{(\ov v, v)}\sum_{j\in V_{n}\setminus\{\ov v\}} (F_{n+1})_{(\ov v, j)}p_{n+1}(\ov v)\\
        &\nonumber\geq (F_n \cdots F_{N+1})_{\ov v, v}f^{(n+1)}_{\ov v, \ov v}\mu(C_{n+1}(\ov v)) \left( 1-\frac{1+\sum_{w\in V_{n}\setminus\{\ov v\}} f^{(n+1)}_{\ov v, w}}{f^{(n+1)}_{\ov v,\ov v}} \right) \\
        &\label{cnn:star}\geq (F_n \cdots F_{N+1})_{\ov v, v}f^{(n+1)}_{\ov v, \ov v}\mu(C_{n+1}(\ov v))(1-\varepsilon)\\
        &\nonumber\geq (1-\varepsilon)(\mu(C_{N}(v)) - \sum_{w'\in V_{n+1}\setminus\{\ov v\}} (F_{n+1} \cdots F_{N+1})_{w', v}\mu(C_{n+1}(w'))\\
        &\nonumber \ -\sum_{w\in V_{n}\setminus\{\ov v\}} (F_n \cdots F_{N+1})_{w,v}f^{(n+1)}_{\ov v, w}\mu(C_{n+1}(\ov v)))\\
        &\nonumber\geq (1-\varepsilon)(\mu(C_{N}(v)) - 2\varepsilon).
    \end{align}
    This finishes the proof.
\end{proof}

\begin{rem}  
    This theorem requires no information about the order of the edges, since in its proof the estimate \eqref{cnn:star} used the worst possible order where every edge incoming to $\ov v$ from another vertex destroys the rigidity of an edge incoming to $\ov v$ from $\ov v$. If we have more information about the ordering, the extra condition of $\frac{\sum_{w \in V_n\setminus \{\ov v\}}f^{(n+1)}_{\ov v, w}}{f^{(n+1)}_{\ov v, \ov v}} \rightarrow 0$ can be weakened or is trivially satisfied, see Corollary~\ref{cor: ERS_odometer_measure_rigid}.
\end{rem}

\begin{thm}\label{Thm:rigid_mu_from_odom}
    Let $B$ be an ordered Bratteli diagram (of arbitrary rank) and $\varphi_B$ be the corresponding Vershik map. Let $\mu$ be an  ergodic invariant probability measure on $B$. Assume that $\mu$ is an extension from a consecutively ordered odometer $\ov B = (\{\ov v\}, \ov E)$ such that
    \begin{equation}\label{eq:limsup_of_number_of_edges_in_odometer}
                \limsup_{n\to\infty} f^{(n)}_{\ov v, \ov v} = \infty.    
    \end{equation}

    %with the growing number of edges on each level and a consecutive order. 
    Then the system $(X_B, \varphi_B, \mu)$ is measure-theoretically rigid. The rigidity sequence is a subsequence of the heights $(h_n(\ov v))_n$. 
\end{thm}

\begin{proof}
    %Without loss of generality we can assume that the measure $\mu$ is an extension of the odometer which passes through the first vertex of the diagram. Indeed, the change of the enumeration of the vertices doesn't change the associated dynamical system or the fact that the order on the diagram is consecutive. 
    Let $a_n = f^{(n)}_{\ov v, \ov v} $. 
    Since $\mu$ is an extension from the odometer which passes through the vertex $\ov v$ on each level of the diagram, the support of $\mu$ consists of all infinite paths that eventually passes through the first vertex. Hence, by Proposition~\ref{Prop:meas_ext_towers}, we have
    $$
    \lim_{n \to \infty }\mu(X_n(\ov v)) = 1 
    \qquad \text{ and } \qquad \lim_{n \to \infty }\mu\left(\bigsqcup_{w \neq \ov v} X_n(w)\right) = 0.
    $$
     Thus, for any cylinder set $C_N(w)$ which ends at vertex $w$ on level $N$ we have
     $$
    \mu(C_N(w)) = \lim_{k \rightarrow \infty} \mu(C_N{(i)} \cap X_{n_k}{(\ov v)} \cap X_{n_{k}+1}(\ov v)) = \lim_{k \rightarrow \infty} (F_{n_k-1}\cdots F_{N+1})_{\ov v,w}a_{n_k}\mu(C_{n_k}(\ov v)).
    $$
    We decompose $C_N(\ov v)$ into paths that stay in the vertex $\ov v$ and paths moving to the other vertices.  %,
    %$$
	  % C_N{(1)} = C_{N,\infty}{(1)} \uplus \bigcup_{j\not= 1} C_{N,\infty}{(j)}.
    %$$
The measure $\mu$ gives zero mass to all paths not tail equivalent to paths inside $\ov B$.
$$
\mu(C_N{(w)}) = \lim_{n \rightarrow \infty} \overline \mu(C_{n}{(\ov v)}) (F_{n}\cdots F_{N+1})_{\ov v,w} = \lim_{n \rightarrow \infty} \frac{(F_{n}\cdots F_{N+1})_{\ov v,w}}{a_1 a_2 \ldots a_{n}}.
$$
Since the order on the diagram is consecutive, we have
    \begin{align*}
    \mu(T^{h_{n_k}(\ov v)}(C_N{(w)}) \cap C_N{(w)}) 
    %&\geq ((a_n-1)(F_{n-1}\cdots F_{N})_{1,1} + (b_n - 1)(F_{n-1}\cdots F_{N})_{2,1})\mu(C_{1}^{(n+1)}) \\&\quad+ ((a_n - 1)(F_{n-1}\cdots F_{N})_{2,1} + (b_n - 1)(F_{n-1}\cdots F_{N})_{1,1})\mu(C_{2}^{(n+1)})\\
    &\geq (a_{n_k}-1)(F_{{n_k}-1}\cdots F_{N})_{\ov v,w}\mu(C_{n_k}{(\ov v)})\\
    &= \frac{a_{n_k} - 1}{a_{n_k}}(F_{n_k-1}\cdots F_N)_{\ov v,w}a_{n_k}\mu(C_{n_k}(\ov v)) \rightarrow \mu(C_N{(w)})
\end{align*}
as $n \rightarrow \infty$. Thus, we proved measure-theoretical rigidity for all cylinder sets and a sequence $t_n = h_n$, hence the system  $(X_B, \varphi_B, \mu)$ is measure-theoretically rigid.
\end{proof}

\begin{rem}
In order for a measure extension from an odometer $\ov B$ to be finite, it is not necessary that \eqref{eq:limsup_of_number_of_edges_in_odometer} holds. For example, for reducible stationary Bratteli diagrams, the measure extension from the stationary odometer $\ov B$ can be finite (see \cite{BezuglyiKwiatkowskiMedynetsSolomyak2010} and Theorem \ref{thm_stat_nonsimple_BD}). It is also not a sufficient condition for a measure extension from $\ov B$ to be finite, since the number of edges may be growing not fast enough
to get the finite measure extension (see \cite{BezuglyiKwiatkowskiMedynetsSolomyak2013}, \cite{AdamskaBezuglyiKarpelKwiatkowski2017} and Proposition \ref{Prop:rank2example}).
\end{rem}

\begin{cor}
 Let $(X_B, \varphi_B, \mu)$ be as in Theorem~\ref{Thm:rigid_mu_from_odom}. Then $(X_B, \varphi_B, \mu)$ has zero entropy.
\end{cor}

Recall that two Cantor minimal systems $(X, T)$ and $(Y,S)$ are called \textit{Kakutani equivalent} if there exist clopen sets $U \subset X$ and $V \subset Y$ such that the induced systems $(U, T_U)$ and $(V, S_V)$ are topologically conjugate (see e.g. \cite{Durand2010}). It was proved in \cite{HermanPutnamSkau1992} that a Bratteli-Vershik dynamical system  $(X_B, \varphi_B)$ associated with a simple properly ordered Bratteli diagram $(B, \geq)$ is Kakutani equivalent to a Cantor minimal system $(Y,S)$ if and only if $(Y,S)$ is topologically conjugate to a Bratteli-Vershik system $(X_{B'}, \varphi_{B'})$, where $(B', \geq')$ and $(B, \geq)$ differ only on finite initial portions (see also \cite{DurandPerrin2022}).

\begin{cor}
    Let $(X_B, \varphi_B, \mu)$ be as in Theorem~\ref{Thm:rigid_mu_from_odom} and additionally $\varphi_B$ be minimal and $B$ be properly ordered. Let $(Y, T)$ be any Cantor minimal system which is Kakutani equivalent  to  $(X_B, \varphi_B)$. Then there is an ergodic invariant probability measure $\nu$ for $(Y, T)$ such that $\nu$ is an extension from an odometer and $(Y, T, \nu)$ is measure-theoretically rigid.
\end{cor}

\subsection{Toeplitz systems as Bratteli-Vershik systems}

Below we introduce some classes of Bratteli diagrams (for more details, see e.g. \cite{AdamskaBezuglyiKarpelKwiatkowski2017}).

\begin{dfn}\label{Def:ERS}
    A non-negative integer matrix $F = (f_{ij})$ satisfies the 
    \begin{itemize}
        \item \textit{equal row sum} property (denoted $F \in ERS$ or $F \in ERS(r)$) if there is  $r \in \N$
such that
    $$
    \sum_{j} f_{ij} = r \qquad \text{ for all } i.
    $$
    \item \textit{equal column sum} property 
    (denoted $F \in ECS$ or $F \in ECS(c)$) if there is  $c \in \N$ such that
    $$
    \sum_{i} f_{ij} = c \qquad \text{ for all } j.
    $$
     \end{itemize}
\end{dfn}

Recall that $h_n{(w)}$ is the number of paths between $v_0$ and $w \in V_n$. If $F_n \in ERS(r_n)$ for all $n$, then it is easy to check by induction that 
$$
h_n(w) = r_1\cdots r_{n}
$$
for every $n$ and every $w \in V_n$. 
If $F_n \in ECS(c_n)$ for all $n$, then there is an invariant probability measure on $X_B$ such that the measures of cylinder sets $C_n(w)$ of length $n$ which end at a vertex $w \in V_n$ are
$$
\mu(C_n(w)) = \frac{1}{c_1\cdots c_{n}}.
$$

In \cite{GjerdeJohansen2000}, the ERS property for incidence matrices $(F_n)_n$ is called the \textit{equal path number property}. The following theorem holds:

\begin{thm}\cite{GjerdeJohansen2000}\label{Thm:GjerdeJohansenToeplitz}
The family of expansive Bratteli-Vershik systems associated to simple properly ordered Bratteli diagrams with the equal path number property coincides with the family of Toeplitz systems up to conjugacy.
\end{thm}

For ERS systems we can use~\ref{Thm:rigid_mu_from_odom_withoutorder} to show the following.

\begin{cor}\label{cor: ERS_odometer_measure_rigid}
 Let $B$ be an ordered Bratteli diagram with incidence matrices $F_n = (f^{(n)}_{v,w})$ which satisfies the $ERS(r_n)$ property. If the ergodic invariant probability measure $\mu$ is an extension from an odometer $\ov B = (\{\ov v\}, \ov E)$, then the system $(X_B, \varphi_B, \mu)$ is measure-theoretically rigid with sequence $(h_n)_n$.
\end{cor}
\begin{proof}
    By~\cite[Theorem 2.1]{AdamskaBezuglyiKarpelKwiatkowski2017} we know that the extension from the subdiagram $\ov B = (\{\ov v\}, \ov E)$ is finite if and only if 
    $$
        \sum_{n=1}^{\infty} \sum_{w \in V_n\setminus \{\ov v\}} \frac{f^{(n+1)}_{\ov v, w}h_n(\ov v)}{\prod_{i = 1}^{n+1}f^{(i)}_{\ov v, \ov v}} < \infty.
    $$
    Thus condition \eqref{eq:qf} is always satisfied as $\prod_{i = 1}^{n}f^{(i)}_{\ov v, \ov v} \leq \prod_{j = 1}^{n}r_j = h_n(\ov v) $ and 
    \begin{align*}
        \frac{\sum_{w \in V_n\setminus \{\ov v\}}f^{(n+1)}_{\ov v, w}}{f^{(n+1)}_{\ov v, \ov v}} \leq \frac{\prod_{j = 1}^{n}r_j\sum_{w \in V_n\setminus \{\ov v\}} f^{(n+1)}_{\ov v, w}}{\prod_{i = 1}^{n+1}f^{(i)}_{\ov v, \ov v}} \longrightarrow 0 \text{ \ as $n\to \infty$.}    
    \end{align*}
\end{proof}

In \cite{Williams84} the following condition for regularity of Toeplitz systems is proven for subshifts, we have adapted it into the Bratteli-Vershik setting.
%\SRnote{check Williams for other direction, can GJ construction force a reps such that sum diverges for all regular Toeplitz systems?}

\begin{prop}\label{prop:length_reg_Toeplitz}
Let $(X,T)$ be a Toeplitz system of finite rank and let $B=(V,E)$ be a Bratteli-Vershik representation of $(X,T)$ by the construction from \cite{GjerdeJohansen2000}. Let $\theta_n$ be the substitution read on level $n$ with constant length $r_n$. If 
$$
    \sum_{n=1}^\infty \frac{1}{r_n} \text{\ diverges,}
$$ 
then the Toeplitz system is regular.
\end{prop}

\begin{proof}
We know that by the construction $\theta_n$ is proper, primitive and of constant length $r_n>2$. In \cite{Kurka2003} the algorithm to go from a constant length, proper and primitive substitution to its Toeplitz sequence $\omega$ is explained.
We define
$p_n = \prod_{i=1}^n r_i$  as the length of substitution words in $\theta_1 \circ \dots \circ \theta_n$
and compute the density up to level $n$ as
$$
    \delta_n = \frac{\delta_{n-1}p_{n-1}(r_n-2)}{p_n} = \delta_{n-1}\frac{r_n-2}{r_n}.
$$
Therefore the density of unknown symbols is
$$
\delta(\omega) = \lim_{n \to \infty} \prod_{i=1}^{n}\left( 1 - \frac{2}{r_i}\right).
$$
Thus the density converges to zero if and only if $\sum_{i=1}^\infty \frac{2}{r_i} =\infty$.
\end{proof}

In case that $\sum_{n=1}^\infty \frac{1}{r_n} <\infty$, there might be another Bratteli-Vershik representation such that the sum diverges. By telescoping an existing Bratteli-Vershik system we can always find a representation such that the sum converges. Thus, Proposition~\ref{prop:length_reg_Toeplitz} provides only a sufficient but not necessary condition for a Toeplitz system to be regular.

\begin{rem}
    From this follows that any Toeplitz system generated by finitely many proper substitutions (in a Bratteli-Vershik context this property is called linearly recurrent) is regular.
\end{rem}

Proposition~\ref{Prop:rank2example} describes ergodic invariant measures and their supports for Bratteli diagrams with $2 \times 2$ incidence matrices satisfying the ERS property (see also \cite{FerencziFisherTalet2009}).

%\tcr{generalize for $n \times n$}
\begin{prop}\cite{AdamskaBezuglyiKarpelKwiatkowski2017}\label{Prop:rank2example}
Let $B$ be a Bratteli diagram with $2 \times 2$ incidence matrices $F_n = (f^{(n)}_{v, w})$ satisfying $ERS$:
\begin{equation}\label{2times2F}
    F_1 = \begin{pmatrix}
    1 \\ 
    1
\end{pmatrix} \quad \text{ and } \quad
F_n =
\begin{pmatrix}
a_n & b_n\\
c_n & d_n
\end{pmatrix}, \;
\text{ where } a_n + b_n = c_n + d_n = r_n \text{ for every } n \geq 2.
\end{equation}

\noindent (i) Let $\ov B (\ov v)$ be an odometer which is a vertex subdiagram of $B$ generated by vertices $\ov v \in V_n$. 
The extension of the unique invariant measure $\ov \mu$ from the odometer $\ov B (\ov v)$ to the path space $X_B$ is finite if and only if
\begin{equation}\label{eq:meas_ext_odom}
\sum_{n=0}^{\infty} \frac{f_{\ov v,v_n}^{(n+1)}}{r_{n+1}} < \infty,
\end{equation}
where $\{v_n\} = V_n \setminus \{\ov v\}$.

\noindent (ii) There are exactly two finite ergodic invariant measures on $B$ if and only if
%\begin{equation}\label{2measures}
%\sum_{k=1}^{\infty}\left(1 - \frac{|a_{k} - c_{k}|}{r_k}\right) < \infty,
%\end{equation}
%or, equivalently,
\begin{equation}\label{2series}
\sum_{k=1}^{\infty}\frac{\min\{b_{k},d_{k}\}}{r_k}< \infty\  \  \mbox{ and }
\  \  \sum_{k=1}^{\infty} \frac{\min\{a_{k},c_{k}\}}{r_k} < \infty.
\end{equation}
In this case, these measures are supported on odometers that satisfy condition \eqref{eq:meas_ext_odom}.

\noindent (iii) %There is a unique invariant measure $\mu$ on $B$ if and only if
%\begin{equation}\label{1measure}
%\sum_{k=1}^{\infty}\left(1 - \frac{|a_{k} - c_{k}|}{r_k}\right) = \infty.
%\end{equation}
%Moreover, 
If
\begin{equation}\label{min}
\sum_{k=1}^{\infty} \frac{\min\{a_{k},b_{k},c_{k},d_k\}}{r_k} = \infty,
\end{equation}
then there is no odometer such that the unique measure $\mu$ would be the extension of a measure supported by this odometer. %Otherwise, if for instance (\ref{1measure}) holds and (\ref{min}) does not hold, then there is an example where the unique invariant measure is an extension of an odometer and there is an example when it is not.

\end{prop}

\begin{rem}
    Let $B$ be as in \eqref{2times2F}. Assume that $B$ has an ergodic invariant probability measure which is an extension from an odometer. After changing the numeration of the vertices, we may always assume that the odometer passes through the first vertex of each level. From \eqref{eq:meas_ext_odom} it follows that
    $$
    \sum_{n=0}^{\infty} \frac{b_n}{r_{n}} < \infty.
    $$
    Assume that $B$ is simple, hence $b_n \neq 0$ for infinitely many $n$. Thus, we have 
    $$
    \limsup_{n\to\infty} r_n = \infty
    $$
    and hence
    $$
    \limsup_{n\to\infty} a_n = \infty.
    $$

\end{rem}

We let $B$ be as in \eqref{2times2F} with a simple hat and investigate its rigidity behaviors.
By Theorems~\ref{Thm:DownarMaass}, \ref{Thm:GjerdeJohansenToeplitz} and \ref{Thm:S-adic_Rep_BD} if the diagram $B$ is properly ordered then it can model either an odometer (if $a_n = b_n$ and the same order for incoming edges in both vertices) or a Toeplitz system ($S$-adic subshift).

Generally there are three possible situations for the ergodic measures. Either there are two ergodic measures, both finite extensions from odometers or there is a unique invariant measure which is either a finite extensions from an odometer or it is of exact finite rank, see Proposition \ref{Prop:rank2example}. If the ergodic measures are finite extensions from an odometer, Proposition \ref{Prop:rank2example} (\ref{eq:meas_ext_odom}) shows that the conditions of Theorem \ref{Thm:rigid_mu_from_odom_withoutorder} are satisfied and thus the measure is rigid. In the case of a unique invariant measure with exact finite rank we need further assumptions such as regularity by $\sum 1/r_n < \infty$ or $a_n = d_n$ to achieve rigidity. For example let $F_n = F$ be stationary with $a, b, c,d \in \N$ such that $a + b= c+d = r$, then the system is uniquely ergodic and by Proposition~\ref{prop:length_reg_Toeplitz} the Toeplitz system is regular. Furthermore by Remark~\ref{Rem:RegToeplitz} its unique invariant measure is rigid.
Another exact finite rank example that is irregular is shown later in Example \ref{ex:rigidToeplitz_2x2_exactfiniterank}.

In the case of a $2 \times 2$ incidence matrix additionally satisfying the ECS property we have the following result.

\begin{thm}\label{Thm:rigid_a_n_b_n_uniquely_erg}
Let $B$ be a non-stationary Bratteli diagram with incidence matrices 
$$
  F_1 = \begin{pmatrix}
    1 \\ 
    1
\end{pmatrix} \quad \text{ and } \quad
F_n =
\begin{pmatrix}
a_n & b_n\\
b_n & a_n
\end{pmatrix}, \  \text{ for } \ n \geq 2.
$$
 such that $a_n + b_n \to \infty$ as $n \rightarrow \infty$. Denote $r_n = c_n = a_n + b_n$ for $n \geq 2$ and $c_1 = 2$ , $r_1= 1$. Then for all $n \geq 1$ both $F_n \in ERS(r_n)$ and $F_n \in ECS(c_n)$. In particular, there exists a probability invariant measure $\mu$ (not necessarily ergodic) on $B$ such that
$$
    \mu(C_n) = \frac{1}{c_1 \cdots c_{n}} \text{ for a cylinder set of length $n$.}
$$
    Endow $B$ with a left-to-right ordering and let $\varphi_B$ be a corresponding Vershik map. Then the system $(X_B, \varphi_B, \mu)$ is measure-theoretically rigid.
\end{thm}

\begin{proof}
We take an arbitrary cylinder set $C_N$ that ends in a vertex $w \in V_N$, then for any $n>N$ we decompose it into cylinders $C_n{(v)}$ of length $n > N$.  There are $(F_{n-1}\cdots F_{N})_{v,w}$ cylinders $C_n{(v)}$ ending in a vertex $v\in V_n$.
Now we take all paths in $C_n{(v)}$ such that the edge $x_{n+1}$ to $\Tilde{v} \in V_{n+1}$ is not the last in the block of edges connecting $v$ to $\Tilde{v}$. These paths will return to the cylinder $C_n{(v)}$ after $h_n(v)$ steps. Thus by the special structure of $F_n$ we get
\begin{align*}
    \mu(T^{h_n}(C_N) \cap C_N) &\geq (a_{n+1}+b_{n+1}-2)((F_{n}\cdots F_{N+1})_{1,w} + (F_{n}\cdots F_{N+1})_{2,w})\mu(C_{n+1})\\
    &=(a_{n+1}+b_{n+1}-2)\frac{(a_{N+1}+b_{N+1})\cdots(a_{n} + b_{n})}{(a_0 + b_0) \cdots (a_{n+1} + b_{n+1})}\\
    & = \frac{a_{n+1}+b_{n+1}-2}{(a_0 + b_0) \cdots (a_{N}+b_{N})(a_{n+1} + b_{n+1})}\\
    & = \frac{a_{n+1}+b_{n+1}-2}{a_{n+1}+b_{n+1}}\mu(C_N).
\end{align*}
Denote
$$
\alpha_n =  \frac{a_n+b_n-2}{a_n+b_n}.
$$
Then $\alpha_n \rightarrow 1$ as $n \rightarrow \infty$. Thus
$$
\lim_{n \rightarrow \infty}\mu(T^{h_n}(C_N) \cap C_N) = \mu(C_N)
$$
for all cylinder sets $C_N$. Thus, the system $(X_B, \varphi_B, \mu)$ is measure-theoretically rigid.
\end{proof}

If by Proposition 3.1 in \cite{AdamskaBezuglyiKarpelKwiatkowski2017} diagram $B$ has a unique ergodic invariant probability measure, then the measure defined in the Theorem is it. Thus in the case of $B$ with exact finite rank, the system $(X_B, \varphi_B, \mu)$ rigid.

\begin{ex}[Rigid Toeplitz system which is not measure-theoretical isomorphic to any odometer]\label{ex:rigidToeplitz_2x2_exactfiniterank}
As in Example 6 from \cite{ArbuluDurandEspinoza2023} we define an S-adic subshift $(X,T)$ on the alphabet for $\mathcal{A} = \{1,2 \}$ with substitutions
$$
\theta_n :
\begin{cases}
1\mapsto (121)^{s(n)}2,\\
2\mapsto 1(121)^{s(n)}
\end{cases}
$$
and $3s(n)+1 = 5^{2n}$. These substitutions are primitive and of constant length. By \cite{ArbuluDurandEspinoza2023} the generated subshift is a Toeplitz system with finite topological rank (thus zero entropy) that does not have a discrete spectrum and is not measure-theoretical isomorphic to its maximal equicontinuous factor, the odometer $(\Z_5, +1)$, or any other odometer.

We can represent this subshift as a Bratteli-Vershik system. As the substitutions $\theta_n$ are not proper the Bratteli diagram has different substitution reads on even and odd levels, see \cite{DurandLeroy2012}
$$
\theta_{2n} :
\begin{cases}
1\mapsto (121)^{s(2n)}2,\\
2\mapsto 1(121)^{s(2n)}
\end{cases}
$$
and
$$
\theta'_{2n+1} :
\begin{cases}
1\mapsto 12(121)^{s(2n+1)-1}21,\\
2\mapsto (121)^{s(2n+1)}1.
\end{cases}
$$
The incidence matrices are
$$
F_n = \begin{pmatrix}
    2s(n) & s(n)+1\\
    2s(n) +1 & s(n)
\end{pmatrix}
$$
with ERS for all $n\in\N$, thus $h_n(1) = h_n(2)= h_n$. The diagram is of exact finite rank and thus has a unique ergodic probability measure $\mu$ by Theorem~\ref{thm:BKMS13ExactFinRank}.
It follows from Theorem~\ref{thm:Danilenko}
%\cite{Danilenko2016} 
that the system is partially rigid. To show measure-theoretical rigidity we prove that $(3h_m)_m$ is a rigidity sequence.

Take any level $m\in\N$ and cylinder set $C_m(1)$, decomposing this into sets of level $m+1$ gives
$$
    \mu(C_m(1))= 2s(m+1)\mu(C_{m+1}(1)) + (2s(m+1) +1)\mu(C_{m+1}(2)).
$$

By the many repetitions of the letter $1$ in $\theta_n(1)$ (or $\theta'_n(1)$) in every third position, we see that all but at most 4 cylinder subsets $C_{m+1}(1)\subseteq C_m(1)$ return to $C_m(1)$
$$
    T^{3h_m}C_{m+1}(1) \subseteq T^{3h_m}C_{m}(1) \cap C_m(1).
$$
Similarly all but at most 2 cylinder subsets $C_{m+1}(2)\subseteq C_m(1)$ return to $C_m(1)$ after $3h_m$-steps. For cylinder sets $C_m(2)$ all but at most 2 cylinder subsets $C_{m+1}(w)\subseteq C_m(2)$ return after $3h_m$-steps.

Thus for arbitrary cylinder set $C_N(v)$ and $\varepsilon > 0$ by Lemma~\ref{lem:height_m mearsure_m+1 to zero} there exists a level $m$ such that $10h_m\mu_{m+1}(w)<\varepsilon$ for all $w\in V_{m+1}$. We decompose $C_N(v)$ into sets of length $m > N$ and then
\begin{align*}
    \mu(T^{3h_m}C_{N}(v) \cap C_{N}(v)) &\geq (F_m \cdots F_{N+1})_{1, v}\mu(T^{3h_m}C_{m}(1) \cap C_{m}(1))\\  
            &\qquad + (F_m \cdots F_{N+1})_{2, v}\mu(T^{3h_m}C_{m}(2) \cap C_{m}(2)) \\
        &\geq  (F_m \cdots F_{N+1})_{1, v}(\mu(C_m(1)) - 4\mu(C_{m+1}(1)) -2\mu(C_{m+1}(2))) \\ 
            &\qquad + (F_m \cdots F_{N+1})_{2, v}(\mu(C_m(2)) - 2\mu(C_{m+1}(1)) -2\mu(C_{m+1}(2))) \\
        &\geq \mu(C_N(v)) - 10h_m\mu(C_{m+1}(w)) \\
        &> \mu(C_N(v)) - \varepsilon.      
\end{align*}
Thus the system is rigid with rigidity sequence $(3h_m)_m$.

\begin{lem}\label{lem:height_m mearsure_m+1 to zero} 
Let $B$ be a Bratteli diagram of exact finite rank and equal row sum such that $r_n\to \infty$. Then there exists a subsequence $(n_k)_k$ such that 
$$
    \lim_{k\to \infty} h_{n_k-1}\mu(C_{n_k}(v)) = 0      
$$
for all cylinder sets $C_{n_k}(v)$ ending in vertices $v\in V_{n_k}$.
\end{lem}
\begin{proof}
    We know by the equal row sum that $h_n = h_n(v) = \prod_{i=1}^{n} r_{i}$ for all $n\in\N$ and $v\in V_{n}$. From Proposition 5.6.(2) in \cite{BezuglyiKwiatkowskiMedynetsSolomyak2013} it follows that there exists a subsequence $(n_k)_k$ such that for all $v\in V_{n_k}$ there is a constant $c_v > 0$ such that
    $$
    \lim_{k\to \infty} c_v\mu(C_{n_k}(v))\prod_{i=1}^{n_k} r_{i} = 1.
    $$
    Therefore
    \begin{align*}
        \lim_{k\to \infty} c_v h_{n_k-1}\mu(C_{n_k}(v)) r_{n_k} = 1
    \end{align*}
    and since $r_{n_k} \to \infty$, we have
    $$
       \lim_{k\to \infty} h_{n_k-1}\mu(C_{n_k}(v)) = 0.
    $$
    The proof is complete.
\end{proof}

\begin{thm}\label{thm:Toeplitz_system_withrepeating_blocks_rigid}
Let $ (X,T)$ be a finite rank Toeplitz system with ergodic measure $\mu$ such that the following properties hold. Let $B$ be a Bratteli diagram representing $(X,T)$ with incidence matrices $F_n = (f^{(n)}_{v, w}) \in ERS(r_n)$ and let $\theta_n$ be the substitution read for level $n$ of $B$.
        
For every level $n \in \N$ there exists $M_n \in\N$ such that:
\begin{itemize}
    \item Let $\rho^{(n)}_{w, v}$ be the number of indices $i\in \{1,\dots, r_n-M_n\}$ such that 
    $$
        v = \theta_n(w)_i = \theta_n(w)_{i+M_n} \text{ for $v\in V_{n-1}, w\in V_n$.}
    $$
    \item There exists subsequence $(n_k)_k$ such that 
    $$
        h_{n_k} (f^{(n_k+1)}_{w', w} - \rho^{(n_k+1)}_{w', w})\mu(C_{n_k+1}(w')) \to 0 \text{\ as $k\to\infty$}       
    $$
    for all $w\in V_{n_k}$ and $w'\in V_{n_k+1}$.
\end{itemize}
Then the system $(X,T)$ is rigid with sequence $(M_{n_k}h_{n_k})_k$.
\end{thm}

\begin{proof}
    Take arbitrary cylinder set $C_N(v)$, $\varepsilon > 0$ and $m$ such that 
    $$
        \sum_{w\in V_m} h_m \sum_{w'\in V_{m+1}}(f^{(m+1)}_{w', w} - \rho^{(m+1)}_{w', w})\mu(C_{m+1}(w')) < \varepsilon.
    $$ 
    We decompose $C_N(v)$ into subsets of length $m$ 
    $$
        \mu(C_N(v)) = \sum_{w\in V_m} (F_m\cdots F_{N+1})_{v, w}\mu(C_{m}(w))
    $$
    and then
    \begin{align*}
        \mu(C_N(v) \cap T^{M_{m}h_m}C_N(v)) &\geq \sum_{w\in V_m} (F_m\cdots F_{N+1})_{v, w}\mu(C_m(w) \cap T^{M_{m}h_m}C_m(w))\\
                                &\geq  \sum_{w\in V_m} (F_m\cdots F_{N+1})_{v, w}\sum_{w'\in V_{m+1}}\rho^{(m+1)}_{w', w}\mu(C_{m+1}(w'))\\
                                %&\geq  \sum_{w\in V_m} (F_m\cdots F_{N+1})_{v, w}\sum_{w'\in V_{m+1}}(f_{m+1}(w, w') - f_{m+1}(w, w') + \rho_{m+1}(w, w'))\mu(C_{m+1}(w'))\\
                                &\geq  \mu(C_N(v)) - \sum_{w\in V_m} h_m \sum_{w'\in V_{m+1}}(f^{(m+1)}_{w', w} - \rho^{(m+1)}_{w', w})\mu(C_{m+1}(w')) \\
                                &\geq \mu(C_N(v)) - \varepsilon.                               
    \end{align*}
    Thus the system $(X,T)$ is rigid with sequence $(M_{n_k}h_{n_k})_k$.
\end{proof}
\begin{cor}
    If the Toeplitz system is of exact finite rank with $r_n \to \infty$ and  $f^{(n_k+1)}_{w', w} - \rho^{(n_k+1)}_{w', w}$ less than some constant for all levels $n_k$, then the system is rigid.
\end{cor}
\begin{proof}
    Under these assumptions and by Lemma \ref{lem:height_m mearsure_m+1 to zero} the conditions of Theorem \ref{thm:Toeplitz_system_withrepeating_blocks_rigid} are satisfied.
\end{proof}
\end{ex}

As an application of Theorem~\ref{Thm:rigid_mu_from_odom}, we get Proposition 6.8 in \cite{DonosoMaassRadic2023} which states that for every $r \geq 1$ there is an $S$-adic subshift such that the number of its ergodic invariant probability measures is $r$ and every ergodic measure is rigid for the same rigidity sequence. Moreover, the following example provides us with a minimal $S$-adic subshift of zero entropy which is Toeplitz and has countably infinitely many ergodic invariant probability measures which are rigid for the same rigidity sequence.

\begin{ex}\label{ex:infrank} 

We present a class of Bratteli diagrams with countably infinite set of ergodic invariant probability measures. It is a slight modification of a class of diagrams presented in Subsection 6.3 of \cite{BezuglyiKarpelKwiatkowski2019}. To construct such diagrams, we let $V_n = \{0, 1, \ldots, n\}$ for
$n = 0, 1, \ldots$, and let
$\{a_n\}_{n = 0}^{\infty}$ be a sequence of natural numbers such that
\begin{equation}\label{property1}
\sum_{n = 0}^{\infty} \frac{n}{a_n} < \infty.
\end{equation}
To define the edge set $t^{-1}(w)$ for every vertex $w$, we
use the following procedure.
 For $w\in V_{n+1}$ such that $w\neq n
+1$, the set $t^{-1}(w)$ consists of $a_n$ (vertical) edges connecting
$w \in V_{n+1}$ with the vertex $w \in V_n$ and $n$ single edges connecting
 $w \in V_{n+1}$ with every vertex $u \in V_n$, $u \neq w$.
For $w=n+1$,  let $t^{-1}(w)$ contain $(a_n - 1)$ edges connecting $w$
 with the vertex $n$ on level $V_n$, two edges connecting $w$ to $u = n-1 \in V_n$ and %one edge 
 $n-1$ single edges connecting $w$ with all other vertices $u = 0,1, \ldots, n-2$ of $V_n$. Then
 $$
 |t^{-1}(w)| = a_n + n
 $$
for every $w \in V_{n+1}$ and every $n = 0,1, \ldots$

The incidence matrices $\tl F_n$ of $B$ have the following form
$$
\tl F_n =
\begin{pmatrix}
a_n & 1  & \ldots & 1 & 1\\
1 & a_n  & \ldots & 1 & 1\\
\vdots & \vdots  & \ddots & \vdots & \vdots \\
1 & 1 & \ldots & a_n & 1\\
1 & 1 & \ldots & 1 & a_n\\
1 & 1 &  \ldots & 2 & a_n-1\\
\end{pmatrix}
$$
for $n = 1,2,\ldots$ 

We observe that the Bratteli diagram defined above admits
a natural order generating the Bratteli-Vershik homeomorphism.
For instance, we can use the left-to-right order which is consecutive (see Definition~\ref{Def:consec_order}).
Then the minimal edge is always an edge 
between $w$ and the vertex $0 \in V_{n}$ and the maximal edge is an edge 
between $w$ and the vertex $n \in V_n$.
It is easy to see that $X_B$ has a unique
 infinite minimal path passing through the vertices $0 \in V_n$, $n \geq 0$
 and a unique infinite maximal path passing through the vertices $n \in V_n
 $, $n \geq 0$. Thus, a Vershik map $\varphi_B \colon X_B \rightarrow
 X_B$ exists and it is minimal. Figure \ref{fig_countOdometers} depicts 
 Bratteli  diagram defined by matrix ${\tl F}_n$. It is known that all minimal Bratteli-Vershik systems with a 
 consecutive ordering have entropy zero (see e.g. \cite{Durand2010}) hence 
 the system that we describe in this subsection has zero entropy.  By Proposition~\ref{Thm:S-adic_Rep_BD}, the system $(X_B, \varphi_B)$ is a minimal $S$-adic subshift defined by the substitutions read on $B$.

 One can repeat the proof given in Proposition 6.3 of \cite{AdamskaBezuglyiKarpelKwiatkowski2017} to show that the diagram $B$ has countably infinitely many ergodic invariant probability measures. These measures can be obtained as extensions of invariant measures from the odometers that pass through the sequences of vertices $\ov w_0 = (0,0,0, \ldots)$, $\ov w_1 = (0,1,1,\ldots)$, $\ov w_2 = (0,1,2,2,\ldots)$, $\ldots$, $\ov w_{\infty} = (0,1,2,3, \ldots)$. By Theorem~\ref{Thm:rigid_mu_from_odom}, the systems $(X_B, \varphi_B, \mu)$ are measure-theoretically rigid for all ergodic invariant probability measures $\mu$. Since each $\tl F_n$ has an equal row sum property, on %a
 each level $n$ all towers have the same height $h_n$. It follows from the proof of Theorem~\ref{Thm:rigid_mu_from_odom} that the heights of the towers $(h_n)_n$ form a rigidity sequence for $(X_B, \varphi_B, \mu)$ for any ergodic invariant probability measure $\mu$. 
 It follows also that for any invariant probability measure $\nu$, the system $(X_B, \varphi_B, \nu)$ is measure-theoretically rigid with the same rigidity sequence $(h_n)_n$.

\begin{figure}[ht]
\unitlength = 0.4cm
\begin{center}
\begin{graph}(28,14)
\graphnodesize{0.4}
% The top vertex
\roundnode{V0}(1,13.5)
% Vertices of the first level
\roundnode{V11}(1,9)
\roundnode{V12}(10,9)

% Edges of the first level
\edge{V11}{V0}
\edge{V12}{V0}
%\bow{V11}{V0}{0.18}
%\edge{V11}{V0}
%\bow{V11}{V0}{-0.21}
%\bow{V12}{V0}{0.07}
%\edge{V12}{V0}
%\bow{V12}{V0}{-0.06}
%\freetext(1.1,11){$\cdots$}
%\freetext(5.6,11){$\cdots$}

%Vertices of the second level
\roundnode{V21}(1,4.5)
\roundnode{V22}(10,4.5)
\roundnode{V23}(19,4.5)

% Edges of the second level
\bow{V21}{V11}{0.18}
\bow{V21}{V11}{-0.21}
\edge{V22}{V11}
\edge{V21}{V12}
\bow{V22}{V12}{0.18}
\bow{V22}{V12}{-0.21}
\bow{V23}{V11}{0.02}
\bow{V23}{V11}{-0.02}
\bow{V23}{V12}{0.06}
\bow{V23}{V12}{-0.07}
\freetext(1.1,6){$\cdots$}
\freetext(10.1,6){$\cdots$}
\freetext(16,6){$\cdots$}

\bow{V21}{V11}{0.28}
\bow{V22}{V12}{0.28}

%Vertices of the third level
\roundnode{V31}(1,0)
\roundnode{V32}(10,0)
\roundnode{V33}(19,0)
\roundnode{V34}(28,0)

% Edges of the third level
\bow{V31}{V21}{0.18}
\bow{V31}{V21}{-0.21}
\freetext(1.1,1.5){$\cdots$}
\edge{V31}{V22}
\edge{V31}{V23}

\edge{V32}{V21}
\bow{V32}{V22}{0.18}
\freetext(10.1,1.5){$\cdots$}
\bow{V32}{V22}{-0.21}
\edge{V32}{V23}

\edge{V33}{V21}
\bow{V33}{V23}{0.18}
\freetext(19.1,1.5){$\cdots$}
\bow{V33}{V23}{-0.21}
\edge{V33}{V22}

\edge{V34}{V21}
\bow{V34}{V23}{0.06}
\freetext(24.9,1.5){$\cdots$}
\bow{V34}{V23}{-0.06}
\bow{V34}{V22}{0.02}
\bow{V34}{V22}{-0.02}

\freetext(1,-1){$\vdots$}
\freetext(10,-1){$\vdots$}
\freetext(19,-1){$\vdots$}
\freetext(28,-1){$\vdots$}

\bow{V31}{V21}{0.28}
\bow{V32}{V22}{0.28}
\bow{V33}{V23}{0.28}

%\freetext(10,-2.5){Figure 1}

\end{graph}
\end{center}
\caption{}\label{fig_countOdometers}
\vspace{0.8 cm}
\end{figure}

\end{ex}

\begin{lem}
 Example~\ref{ex:infrank} is expansive and hence a Toeplitz shift.
\end{lem}

\begin{proof}
We observe the substitution read $\theta_n$ of $B$ from $V_n$ to $V_{n-1}$ is
\begin{align*}
    \theta_n(j) &= 01\cdots (j-1)j^{a_n}(j+1)\cdots (n-2)(n-1)  \ \text{ for all $j\in V_{n-1} $,}\\
    \theta_n(n) &= 01 \cdots (n-2)^2(n-1)^{a_n-1}.
\end{align*}
and that the all maximal incoming edges have the vertex $n-1$ as its source.

For two edge-coded paths $x$ and $x'$ in the diagram, we define $H(x,x') = \min\{ j \geq 1 : x_j \neq x'_j\}$. Let $\delta > 0$ be such that $d(x,x') > \delta$ if $H(x,x') \leq 2$. Take two distinct such paths $x,x'$ and assume from now on that $k := H(x,x') \geq 3$. Then $s(x_k) = s(x'_k)$. If both $x$ and $x'$ represent non-maximal paths from $v_0$ to $t(x_k)$, and $t(x'_k)$ respectively, then $H(\varphi_B(x), \varphi_B(x')) = H(x,x')$, so we can iterate the Vershik map until one of the paths, say $x$, is maximal between $v_0$ and $t(x_k)$. Then by the previous observation $s(x_k) = k-1$.
We distinguish three cases:
 \begin{enumerate}
  \item $t(x_k) = t(x'_k)$. Since $x_k$ is a maximal incoming edge, this can only be if $t(x_k) = k$ or $k-1$, as otherwise the definition of $k$ is violated. In both cases $s(\varphi_B(x)_k) = 0$ and $s(\varphi_B(x')_k)  = k$ or $k-1$. Hence $H(\varphi_B(x), \varphi_B(x')) < H(x,x')$.
  \item $t(x_k) \neq t(x'_k)$ and $x'_k$ is not a maximal incoming edge. Then $s(\varphi_B(x)_k) = 0$ and $s(\varphi_B(x')_k) \neq 0$ as all outgoing edges of the first vertex are either minimal (so preceded by maximal edges) or preceded by edges $e_k$ with $s(e_k) = 0$, so we never have $s(x'_k)= k-1$, excluding this case as well. Then $H(\varphi_B(x), \varphi_B(x')) < H(x,x')$.
 \item $t(x_k) \neq t(x'_k)$ and $x'_k$ is a maximal incoming edge. Then there exists some multiple $m\in\N$ for the height $h = h_{k-1}(s(x_k))$ such that $s(\varphi_B^{-mh}(x)_k) \neq s(\varphi_B^{-mh}(x')_k)$, so that $H(\varphi_B^{-mh}(x), \varphi_B^{-mh}(x')) < H(x,x')$.
 \end{enumerate}
Hence, there is always some iterate $n \in \Z$ such that $H(\varphi_B^n(x), \varphi_B^n(x')) < H(x,x')$,
and by induction this means that $\liminf_{n \in \Z} H(\varphi_B^n(x), \varphi_B^n(x')) \leq 2$, so $\limsup_{n \in \Z} d(\varphi_B^n(x), \varphi_B^n(x')) > \delta$ and expansivity follows.
\end{proof}

\medskip
\section{Examples of non-rigid partially rigid systems}\label{sec:enum_systems}

% \subsection{Non-miniml example of a non-rigid partially rigid system.}
In this section we give two examples of non-rigidity, both defined by stationary Bratteli diagrams. One is an invertible non-minimal system, the other based on enumeration systems is not invertible.

\subsection{A family of non-minimal examples.}
Using the methods from the previous sections, we
first give a family of stationary Bratteli-Vershik systems with a non-rigid, partially rigid fully supported measure.

\begin{figure}
\unitlength=1cm
\begin{graph}(5,5)
% \graphnodesize{0.2}
 \roundnode{V0}(3,4)
%  %\nodetext{V0}(-1,0){$V_0$}
  % The first level vertices
 \roundnode{V11}(2,3)
 \roundnode{V12}(4,3)
 
 % The second level vertices
 \roundnode{V21}(2,1)
 \roundnode{V22}(4,1)
  
 %
 % EDGES
 \graphlinewidth{0.025}
% % First level
 \edge{V0}{V11}
 \edge{V0}{V12}

 % Second level
 \bow{V21}{V11}{0.09}
  \bow{V21}{V11}{-0.09}
  \bowtext{V21}{V11}{0.09}{0}
  \bowtext{V21}{V11}{-0.09}{1}
    \bow{V22}{V11}{0.06}
    \bow{V22}{V11}{-0.06}
    
 \bowtext{V22}{V11}{0.06}{0}
 \bowtext{V22}{V11}{-0.06}{5}
     
 \bow{V22}{V12}{0.05}
 \bow{V22}{V12}{-0.05}
  \bow{V22}{V12}{0.12}
 \bow{V22}{V12}{-0.12}

  \bowtext{V22}{V12}{0.17}{1}
 \bowtext{V22}{V12}{-0.17}{4}
   \bowtext{V22}{V12}{0.06}{2}
 \bowtext{V22}{V12}{-0.06}{3}
 
\freetext(2,0.5){$\vdots$}
\freetext(4,0.5){$\vdots$}
%\freetext(3,0.1){.\,.\,.\,.\,.\,.\,.\,.\,.\,.\,.\,.\,.\,.\,.\,.\,.\,.\,.\,.\,.}
\end{graph}
\caption{}\label{Fig:NonRigid}
\end{figure}

\begin{thm}\label{thm_stat_nonsimple_BD}
Let $p \geq 2$ and $q > 2$ be integers and $B$ be the stationary Bratteli diagram with the incidence matrix
$$
F = \begin{pmatrix}
2 & 0\\
p & q
\end{pmatrix},
$$
so that the vertical odometer corresponding to the second vertex is consecutively ordered  and the maximal and minimal edges that end in the second vertex start at the first vertex. Let $\varphi_B$ be the corresponding Vershik homeomorphism.
%as in Figure~\ref{Fig:ConsecutiveOrdering???}
Then $(X_B, \varphi_B)$ is partially rigid but not rigid w.r.t.\ the unique fully supported ergodic measure.
\end{thm}

\begin{rem}\label{Rem:NonRigid}
Figure~\ref{Fig:NonRigid} demonstrates the order for $p = 2$ and $q = 4$.
In fact, the same conclusion holds for any stationary choice of order such that the vertical odometer corresponding to the second vertex is consecutively ordered, but for simplicity of the proof, we fixed the order as we did,
with a unique minimal and a unique maximal path.    
\end{rem}

%\\[3mm]
\begin{proof}
This Bratteli-Vershik systems is transitive but not minimal.
The heights of the towers satisfy
$$
h_n{(1)} = 2^{n - 1}, \qquad h_n{(2)} =
q^{n-1}\left(1 + \frac{p}{q-2}\right) - \frac{2^{n-1}p }{q-2} =
q h_{n-1}(2) + p h_{n-1}(1). % \text{ and } h_1(2) = 1.
$$
The diagram preserves exactly two finite ergodic measures, $\mu_1$ and $\mu_2$ (see \cite{BezuglyiKwiatkowskiMedynetsSolomyak2010}).
The measure $\mu_1$ is supported on the odometer subdiagram $\ov B$ with vertex $v_n(1)$. By Theorem \ref{Thm:rigid_mu_from_odom_withoutorder} the measure $\mu_1$ is rigid with rigidity sequence $(h_n(1))_n$. However, this measure is not fully supported.
The other measure $\mu_2$ is fully supported, because each tail-equivalence class that eventually only goes through $v_n(2)$ is dense in $B$. 

First, we will show that $(X,\varphi_B,\mu_2)$ is partially rigid.
Fix $N \geq 1$ and for each $n \geq N$, let $C_n$ be a cylinder set of length $n$ that ends at $t(x_n) = v_n(2)$. 
We partition $C_n$ into $(n+1)$-cylinders $C_{n+1}(1), \dots, C_{n+1}(q)$ according to the edge $x_{n+1}$; each $C_{n+1}(a)$, $1 \leq a < q$, surely returns to $C_n$ after $s_n = h_n{(2)}$ steps.
Therefore
$$
	\liminf_{n\to \infty} \mu_2(\varphi_B^{s_n}(C_N) \cap C_N) \geq \frac{q-1}{q}\mu(C_N)
$$
for any cylinder $C_N$ ending in the second vertex.

In contrast, consider an $N$-cylinder that ends at the first vertex of the diagram. We decompose it into paths that stay at the first vertex and paths moving to the second vertex. The measure $\mu_2$ gives zero mass to paths that stay at the first vertex for all levels $n$. The set of paths moving to the second vertex at some point is a countable union of subcylinders ending at the second vertex. 
Thus the system $(X, \varphi_B, \mu_2)$ is partially rigid with $\alpha \geq \frac{q-1}{q}$.

Now for the non-rigidity of $\mu_2$, take $\eps \in (0,(10q)^{-6})$ arbitrary, and
take $n_0$ maximal such that $2^{-n_0} > 4\eps$.
Therefore any two distinct $n_0$-cylinders
are at least $4\eps$ apart.
Also, take $n_1 \in \N$ minimal so that $h_{n_0}(2) < p h_{n_1+1}(1)$,
which also means $p h_{n_1+1}(1) < 2h_{n_0}(2)$ because $h_{n_1+1}(1) = 2h_{n_1}(1)$. Without loss of generality let us assume that for all $n$, the vertical edges with the source $v_n(2)$ and target $v_{n+1}(2)$ are enumerated from $1$ to $q$ among all edges in $t^{-1}(v_{n+1}(2))$; cf. Remark~\ref{Rem:NonRigid}.
For $n \geq 1$ and $0 \leq a < p+q$, set 
$$
Y_{n+1}(a) := \{ x \in X_B : t(x_{n+1}) = v_{n+1}(2) \ \text{ and }\ x_{n+1} = a\}
\text{ and }
Z_{n+1}(a) = Y_{n+1}(q) \cap Y_{n+2}(a).
$$
{\bf Claim 1:} 
The sequence $(h_n(2))_{n \geq 1}$ is not a sequence of rigidity times. More precisely, for all $n > n_1$ and $1 \leq a < q$:
$$
\mu_2(\{ x \in Z_{n+1}(a)  : d(\varphi_B^{h_n(2)}(x),x) > 4\eps\}) > \frac{1}{(q+1)^2} \mu_2(Z_{n+1}(a)).
$$
\begin{proof}[Proof of Claim 1]%\let\qed\relax
To prove the claim, let $n > n_1$ and $C \subset Z_{n+1}(a)$  be a cylinder with edges $x_1 \dots x_{n+2}$,
$x_j \in \{1,\dots,q\}$ for $2 \leq j \leq n$, and in particular
$x_{n_0+1} = q-1$, $x_{n_1+1} =  1$,
$x_{n+1} = q$ and $x_{n+2} \in \{1,\dots,q-1\}$.
Since $p\mu_2(Y_n(0)) = p\mu_2(Y_n(q+1)) = \dots =  p\mu_2(Y_n(q+p-1)) < \mu_2(Y_n(1)) = \mu_2(Y_n(2)) = \dots = \mu_2(Y_n(q))$
the collection of such cylinders $C$ has at least $1/(q+1)^2$ of the mass of $Z_{n+1}(a)$.
Let $C' = \varphi_B^{h_{n_1}(2)}(C)$, so for every
$x \in C$ and $x' = \varphi_B^{h_{n_1}(2)}(x)$ satisfies $x'_{n_1+1} = 2$, and $x'_j = x_j$ for all $j \neq n_1+1$,  in particular $C' \subset Z_{n+1}(a)$.

It takes $h_j(2)$ iterates to change $x_{j+1} \in \{1,\dots,q-1\}$ to $x_{j+1}+1$ and restore all the edges $x_k$, $k \leq j$. Similarly, it takes
$h_j(2)+p h_j(1)$ iterates to change $x_{j+1} = q$ to edge $1$
and restore all the edges $x_k$, $k \leq j$, provided that $x_{j+2}\neq q$.
We compute inductively
\begin{eqnarray*}
h_n(2) - p h_n(1) &=&
 q h_{n-1}(2) + p h_{n-1}(1) - p h_n(1) \\
 &=& (q-1) h_{n-1}(2) + h_{n-1}(2) - p h_{n-1}(1) +
 p(2 h_{n-1}(1) -  h_n(1)) \\
 &=& (q-1) h_{n-1}(2) + h_{n-1}(2) - p h_{n-1}(1) \\
&=&  (q-1) h_{n-1}(2) + \cdots + (q-1) h_1(2) +
h_1(2) - p h_1(1) \\
&=&  (q-1) \sum_{j=1}^{n-1} h_j(2) +
1 - p \\
&=& (q-1) \sum_{j=1}^{n-1} h_j(2) + p \sum_{j=1, x_{j+1} \neq 1}^{n-1} h_j(1) - d,
\end{eqnarray*}
where $d = p \sum_{j=1, x_{j+1} \neq 1}^{n-1} h_j(1) + p-1$.
This means that
$\varphi_B^{h_n(2)}(x) = \varphi_B^{-d}(y)$,
for a path $y$ with 
%$y_j \equiv x_j + q-1 \bmod q \in \{1,\dots,q\}$ 
$$
y_j = \begin{cases} 
    x_j - 1 & \text{ for } x_j > 1,\\
    q  & \text{ for } x_j = 1
    \end{cases}
$$
for all $2 \leq j \leq n$, $y_{n+1}=1$, $y_{n+2}=x_{n+2}+1$ and $y_j=x_j$ for $j>n+2$.

Let $y'$ be the analogue for $x' = \varphi_B^{h_{n_1}(2)}(x)$,
so $\varphi_B^{h_n(2)}(x') = \varphi_B^{-d'}(y')$, but
note that $d' = d+p h_{n_1+1}(1)$.

Now if $\varphi_B^{-d}(y)_j \neq x_j$ for some $j \leq n_0$, then
$d(\varphi_B^{h_n(2)}(x),x) > 4\eps$. So assume that
$\varphi_B^{-d}(y)_j = x_j$ for all $j \leq n_0$.
Then also $\varphi_B^{-d}(y')_j = x'_j$ for all $j \leq n_0$.
But recall that
$$
d'-d = p h_{n_1+1}(1) = h_{n_0}(2) + r
\quad \text{ for some } 0 < r < h_{n_0}(2).
$$
Therefore, taking $z' = \varphi_B^{-h_{n_0}(2)} \circ \varphi_B^{-d}(y')$ we find
$$
\varphi_B^{h_n(2)}(x') %= \varphi_B^{-d'}(y')
= \varphi_B^{-r}(z')
\quad \text{ and } \quad z'_j = \begin{cases} x'_j & \text{ for } j \leq n_0,\\
                x'_{n_0+1} - 1 = q-2 & \text{ for } j= n_0+1.
                                \end{cases}
$$
Therefore $[\varphi_B^{-h_{n_1}(1)}(z')]_j \neq x'_j$
for at least one $j \leq n_0$.
This shows that $d(\varphi_B^{h_n(2)}(x'),x') > 4\eps$.

We obtain that, at least one of $C$ and $C'$ does not return to itself after $h_n(2)$ iterates. Recall that $x_{n_1+1} =  1$ while $x'_{n_1+1} =  2$, hence the above lack of rigidity applies to cylinders that represent at least $1/(q+1)^2$ of the mass of $Z_{n+1}(a)$, proving the claim.
\end{proof}

\iffalse
\medskip
{\red {\bf HB:} I noticed a mistake in the argument below: If $s > q h_n(2)$,
there is a problem. I was in the middle of some fix, but now I notice that that one doesn't work either. So right now the rest of the proof is in disarray.}
\medskip

The claim has the following easy corollary: For each $1 \leq q' < q$,
$(q'h_n(2))_{n \geq 1}$ is not a sequence of rigidity times, or more precisely, for all $n > n_1$:
\begin{equation}\label{eq:corclaim}
\mu_2(\{ x \in Z_{n+1}(a) : d(\varphi_B^{q'h_n(2)}(x),x) > 3\eps\}) > \frac{1}{2(q+1)^2} \mu_2(Z_{n+1}(a)).
\end{equation}
Indeed, let $Z'_{n+1}$ be the set of $x \in Z_{n+1}(a)$ such that 
$d(\varphi_B^{h_n(2)}(x),x) > 4\eps\})$.
Since $\mu_2(\{ x \in X_B : t(x_n)=v_n(1) \}) < \eps < \frac{1}{2(q+1)^2}$, it follows that
$\mu_2( \{ x \in Z'_{n+1} : t(\varphi_B^{h_n(2)}(x)_n) = v_n(2) \}) > \frac{ \mu_2(Z_{n+1})(a)}{(q+1)^2} - \eps> \frac{ \mu_2(Z_{n+1}(a)) }{2(q+1)^2}$.
But for such a point, 
$$
d(\varphi_B^{q'h_n(2)}(x), x) \geq d(\varphi_B^{h_n(2)}(x), x) 
- d(\varphi_B^{q' h_n(2)}(x), \varphi_B^{h_n(2)}(x)) > 4\eps - \eps = 3\eps.
$$
This proves \eqref{eq:corclaim}
\medskip
\fi

Now to prove the non-rigidity of $\mu_2$, recall that for every $n > n_1$, any two points in the same $n$-cylinder are less than $\eps$ apart.
Suppose by contradiction that $s \geq h_{n_1}(2)$ is a rigidity time in the sense that
$$
\mu_2(U_\eps(s)) > 1-\eps \quad \text{ for } \quad U_\eps(s) = \{ x \in X_B : d(\varphi_B^s(x),x) < \eps\}.
$$
Let $n \geq n_1$ be maximal such that $h_n(2) \leq s$.
First assume that $s < (q-1) h_n(2)$ and $1 \leq a < q$, and 
let $W_{n+1} := \{ x \in Z_{n+1}(a) : d(\varphi_B^{h_n(2)}(x),x) > 4\eps)\}$.
Consider $W'_{n+1} = \varphi_B^{-s}(W_{n+1}) \subset \bigcup_{a=1}^{q-1}Y_{n+1}(a)$ and
$W''_{n+1} = \varphi_B^{h_n(2)}(W'_{n+1}) \subset \bigcup_{a=2}^{q}Y_{n+1}(a)$.
Note that $\mu_2(W''_{n+1}) \geq \frac{1}{(q+1)^2} \mu(Z_{n+1}(a)) > 2\eps$.

If $x \in U_\eps(s) \cap W'_{n+1}$, then
$\varphi_B^{s-h_n(2)}(x) \in \bigcup_{a=1}^{q-1}Y_{n+1}(a)$ and
$$
d(\varphi_B^{s-h_n(2)}(x), x) \leq d(\varphi_B^s(x),x) + d(\varphi_B^s(x), \varphi_B^{s-h_n(2)}(x)) < \eps + \eps = 2\eps.
$$
But then $x' := \varphi_B^{h_n(2)}(x) \in W''_{n+2}$ and
$x'' := \varphi_B^{s-h_n(2)}(x') \varphi_B^s(x)\in W_{n+1}$ 
satisfy
\begin{eqnarray*}
  d(\varphi_B^s(x'), x') &\geq& d(\varphi_B^s(x'),\varphi_B^{s-h_n(2)}(x')) -   d(\varphi_B^{s-h_n(2)}(x'),x') \\
  &=& d(\varphi_B^{h_n(2)}(x''),x'') -   d(\varphi_B^{s-h_n(2)}(x),x) > 4\eps - 2\eps = 2\eps.
\end{eqnarray*}
Hence $W''_{n+1} \cap U_\eps(s) = \emptyset$, which contradicts that $\mu_2(U_\eps(s)) > 1-\eps$.
%{\blue Note that if $x\in W''_{n+1}$ then if we put $y=\varphi_B^{-h_n(2)}(x)$ then $y\in \bigcup_{a=1}^{q-1}Y_{n+1}(a)$ and so in particular $x_j=y_j$ for $j\leq n$. This implies that $d(y,f^{s+h_n(2)}(y))=d(y,f^s(x))>\eps$ because $x_j\neq f^s(x)_j$ for some $j\leq n_0$. Therefore also $\mu_2(U_\eps(s+h_n(2)) \leq 1-\eps$. {\red So we must have $\mu_2(W''_{n+1} \setminus U_\eps(s))<\eps$}}
\\[3mm]
Now $s<2h_n(2)$ so that $h_{n+1}(2) - s \geq (q-1)h_n(2)$.
For $n \geq 1$ and $0 \leq a < p+q$, set 
\begin{align*}
Z_{n+2}'(a) := \{ x \in X_B : t(x_{n+2}) = v_{n+2}(2), \ x_{n_0+1} = q-1, x_{n_1+1} \in \{  1,2\},\\
 x_{n+1}\in \{1,\dots,q-2\}, x_{n+2} = q, x_{n+3} \in \{1,\dots,q-1\}\}
\end{align*}
{\bf Claim 2:} Set $Q=\{ x \in Z'_{n+2}(a)  : d(\varphi_B^{h_{n+1}(2)}(x),x) > 4\eps\}$.
There is $\gamma>(q+1)^{-4}$ such that 
$$
\mu_2(Q) > \frac12\, \mu_2(Z_{n+2}'(a))>\gamma
\qquad \text{ for every } n>n_1.
$$
The proof of Claim 2 is similar to that of Claim 1. The details are left to the reader.

Let $R=\{y\in X_B : y_j=x_j \text{ for }j\neq n+2, y_{n+2}=1 \text{ for some } x\in Q\}$.
Note that since $y_{n+2}=1$ for every $y\in R$
we have $d(\varphi_B^{h_{n+1}(2)}(y),y)<\eps$.
Now fix any $x\in Q$ and associated $y\in R$ which has all coordinates, but $(n+2)$-th the same as $x$.
Then, since $x_{n+1}<q-1$, points $\varphi_B^s(x)$ and $\varphi_B^s(y)$ satisfy $\varphi_B^s(x)_j=\varphi_B^s(y)_j$ for $j\leq n$. Note also that 
$\varphi_B^{h_{n+1}(2)}(y)_j=y_j$ for all $j\leq n_0$
and $\varphi_B^{h_{n+1}(2)}(x)_j\neq x_j$ for some $j\leq n_0$. It means that there is $j\leq n_0$ such that
$\varphi_B^{h_{n+1}(2)}(y)_j\neq \varphi_B^s(y)_j$ or $\varphi_B^{h_{n+1}(2)}(x)_j\neq \varphi_B^s(x)_j$. This in turn means that
$$
d(\varphi_B^s(x),\varphi_B^{h_{n+1}(2)-s}(\varphi_B^s(x)))\geq \eps \text{ or }
d(\varphi_B^s(y),\varphi_B^{h_{n+1}(2)-s}(\varphi_B^s(y)))\geq \eps.
$$
As a consequence $\mu_2(U_\eps(h_{n+1}(2)-s))<1-\gamma$
which is a contradiction.
\end{proof}

 %The odometer in the second vertex grows very fast, don't care about first vertex (component), measure sits on the fast growing odometer. 
\begin{rem}

%If the measure sits on an odometer that grows fast enough then we have measure-theoretical rigidity. 
%\tcr{See Keane's example?}
%\tcr{See different orders (strong orbit equivalence)}
It seems that the assumption that the Bratteli diagram in Theorem \ref{thm_stat_nonsimple_BD} is stationary is important for the result to hold. To see this, 
%We can 
modify the stationary Bratteli diagram from Theorem \ref{thm_stat_nonsimple_BD} to the non-stationary one with the following incidence matrices:
% $$
% F_n = \begin{pmatrix}
%     2 & 0\\
%     2^n - 2 & 2^n
% \end{pmatrix}.
% $$
% \SRnote{to use the rigidity theorem we need $f_{2,2}^{(n)}$ growing and the measure extension of the odometer to be finite, so do we have any problems here with $p > q$? }
% \HBnote{I think we can allow $p > q$ here because 
% $h_n(2)> f_{2,2}^{(n)} = q^{n(n+1)/2}$ is still much bigger than $h_n(1) = p^n 2^{n-1}$.
% After all, also in Theorem~\ref{thm_stat_nonsimple_BD} there is no condition that $p < q$.
% I think we can take 
$$
F_n = \begin{pmatrix}
    2 & 0\\
    p_n & q_n
\end{pmatrix},
$$
where $\limsup_{n\to\infty} q_{n} = \infty$ and  
% as long as $p_n 2^{n-1} = o(\prod_{k=1}^n q_k)$.
the series $\sum_{n=1}^{\infty} \frac{p_n 2^{n-1}}{\prod_{k=1}^n q_k}$ converges.
% $$
% F_n = \begin{pmatrix}
%     2 & 0\\
%     p^n & q^n
% \end{pmatrix}.
% $$ 
Define the order as in Theorem \ref{thm_stat_nonsimple_BD}, then both ergodic invariant measures are rigid (the measure extension from the second odometer is finite by \cite[Theorem 2.1]{AdamskaBezuglyiKarpelKwiatkowski2017}). To prove that the measure sitting of the second vertex is rigid, use Theorem \ref{Thm:rigid_mu_from_odom}. Since the odometer corresponding to the second vertex is consecutively ordered and has the growing number of edges, the sequence $(h_n(2))$ is a rigidity sequence.
\end{rem}

\subsection{Examples from kneading theory and enumeration systems.}

The non-rigid example in this subsection is minimal but not a homeomorphisms, it is one of a family of so called enumeration systems, which exhibit varying rigidity behaviors. 

The following approach and notation comes from kneading theory, i.e., the symbolic dynamics of unimodal maps, see \cite[Section 3.6.3]{Bruin_book}.
Given a {\em kneading map} $Q: \N \to \N \cup \{ 0 \}$ satisfying $Q(k) < k$ for all $k \in \N$, we build a recursive sequence
\begin{equation}\label{eq:ct}
S_0 = 1, \qquad S_k = S_{k-1} + S_{Q(k)}.
\end{equation}
This is the sequence of {\em cutting times}, and it fully determines the combinatorial properties of the orbit $\orb(c)$ of the critical point $c=0$ of a unimodal map $f_{\ell,a}:[0,1] \to [0,1]$, $x \mapsto a - |x|^\ell$.
Here $\ell$ is the so-called critical order and parameter $a$ can be chosen that $f_{\ell,a}$ has cutting times $(S_k)_{k \geq 0}$.\footnote{Provided
the kneading map satisfies Hofbauer's admissibility condition, see \cite[Formula (3.22)]{Bruin_book}.}
For instance, if $Q(k) \to \infty$, then $\overline{\orb(c)}$ is a minimal Cantor set
\cite[Theorem 4.120]{Bruin_book}, and if $\sup_k k-Q(k) < \infty$,
then $f:\overline{\orb(c)} \to \overline{\orb(c)}$ is uniquely ergodic \cite[Corollary 6.41]{Bruin_book}, and a fortiori,
$\orb(c)$ attract Lebesgue-a.e.\ orbit, see \cite[Theorem A]{Bruin_TAMS}.

One can describe the associated dynamics by means of {\em enumeration scales},
introduced in \cite{BaratDownarowiczIwanikLiardet2000, BaratDownarowiczLiardet2002}, see also \cite[Section 5.3]{Bruin_book}. It allows a more general approach to Cantor system, based on more general recursive sequences, but in this section we stick to the recursion \eqref{eq:ct}.
Every $N\in\N_0$ can be represented as a sequence $(x_i)_{i=0}^\infty$
of digits $x_i \in \{ 0, 1\}$ through the greedy algorithm such that
$$
    \langle N\rangle = x_0x_1\cdots \ \ \text{ for }\ \ N = \sum_{i\geq 0} x_i S_i.
$$
We define the set $X_0 = \langle \N_0 \rangle$ and take its closure $X$ in the product topology on $\N_0^{\N_0}$.
%As dynamics the 'addition of one' maps $T: X_0 \mapsto X_0$ by $T(\langle n\rangle) = \langle n+1 \rangle$. This can be continuously extended to the set $X$ and is surjective and minimal. The system $(X, T)$ is called an enumeration system.

As dynamics the 'addition of one' maps $T: X_0 \mapsto X_0$ by $T(\langle n\rangle) = \langle n+1 \rangle$. This can be continuously extended to the set $X$,
provided that $Q(k) \to \infty$ (which is the assumption that we will pose from now on), and in this case $T: X \to X$ is surjective and minimal.
The system $(X, T)$ is called the {\em enumeration system} of the enumeration scale $(S_k)_{k \geq 0}$.

Such enumeration scales can be represented as Bratteli-Vershik systems as follows:
For $i \geq 1$, the vertex set $V_i$ consists of $1+\# K_i$
vertices for $K_i := \{k : Q(k) < i < k\}$, indexed $i$ and $K_i$,
and $V_0 = \{ v_0\}$.
The hat is simple, i.e., $E_1$ contains a unique edge from $v_0$
to each $v \in V_1$.
For $i \geq 1$, $E_{i+1}$ contains an edges (labeled $0$) from $v_i \in V_i$
to $v_{i+1} \in V_{i+1}$ and from $v_i \in V_i$ to $v_k \in V_{i+1}$ if $k > i+1 > i = Q(k)$. If $Q(i+1) = i$, then there is another edge from $v_i \in V_i$ to $v_{i+1} \in V_{i+1}$, but this edge is labeled $1$.
The remaining edges are from $v_k \in V_i$, $k > i$, to $v_k \in V_{i+1}$ and these are all labeled $0$. The labels indicate the (left-to-right) order of incoming vertices.
The resulting Bratteli diagram has a simple hat and a single {\em spine}
$v_0 \stackrel{0}{\to} v_1  \stackrel{0}{\to} v_2  \stackrel{0}{\to} v_3  \stackrel{0}{\to} \dots$ (this is the unique minimal infinite path), and for each $i \geq 2$,
there is a path from $v_i$ upward to $v_{Q(i)} \in V_{Q(i)}$ of which the lowest edge is labeled $1$ and the others are labeled $0$.

\begin{figure}
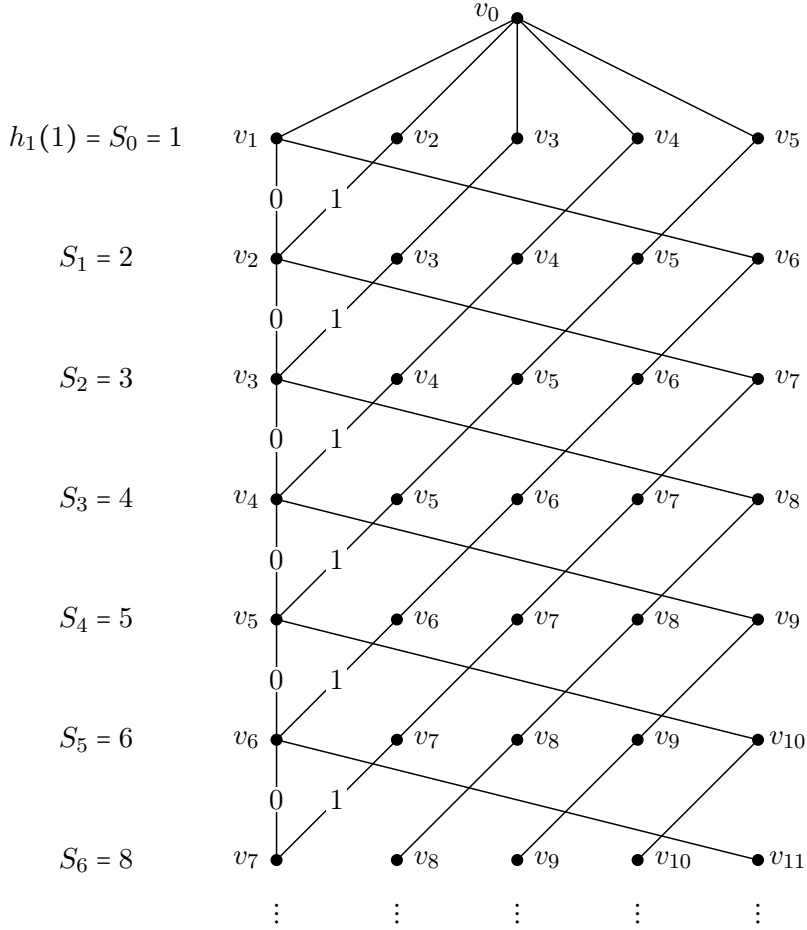

\unitlength=0.8cm
\begin{graph}(10,16)
% \graphnodesize{0.2}

\roundnode{V0}(5,15)
\nodetext{V0}(-0.5,0.1){$v_0$}

  % The first level vertices
 \roundnode{V10}(1,13) 
 \roundnode{V11}(3,13)
 \roundnode{V12}(5,13)
 \roundnode{V13}(7,13)
  \roundnode{V14}(9,13)
  
 \nodetext{V10}(-0.5,0){$v_1$}
 \nodetext{V11}(0.5,0){$v_2$}
 \nodetext{V12}(0.5,0){$v_3$}
 \nodetext{V13}(0.5,0){$v_4$}
  \nodetext{V14}(0.5,0){$v_5$}

 % The second level vertices
 \roundnode{V20}(1,11) 
 \roundnode{V21}(3,11)
 \roundnode{V22}(5,11)
 \roundnode{V23}(7,11)
  \roundnode{V24}(9,11)

   \nodetext{V20}(-0.5,0){$v_2$}
 \nodetext{V21}(0.5,0){$v_3$}
 \nodetext{V22}(0.5,0){$v_4$}
 \nodetext{V23}(0.5,0){$v_5$}
  \nodetext{V24}(0.5,0){$v_6$}

 % The third level vertices
 \roundnode{V30}(1,9) 
 \roundnode{V31}(3,9)
 \roundnode{V32}(5,9)
 \roundnode{V33}(7,9)
  \roundnode{V34}(9,9)

     \nodetext{V30}(-0.5,0){$v_3$}
 \nodetext{V31}(0.5,0){$v_4$}
 \nodetext{V32}(0.5,0){$v_5$}
 \nodetext{V33}(0.5,0){$v_6$}
  \nodetext{V34}(0.5,0){$v_7$}

 % The fourth level vertices
 \roundnode{V40}(1,7) 
 \roundnode{V41}(3,7)
 \roundnode{V42}(5,7)
 \roundnode{V43}(7,7)
  \roundnode{V44}(9,7)

\nodetext{V40}(-0.5,0){$v_4$}
 \nodetext{V41}(0.5,0){$v_5$}
 \nodetext{V42}(0.5,0){$v_6$}
 \nodetext{V43}(0.5,0){$v_7$}
  \nodetext{V44}(0.5,0){$v_8$}

   % The fifth level vertices
 \roundnode{V50}(1,5) 
 \roundnode{V51}(3,5)
 \roundnode{V52}(5,5)
 \roundnode{V53}(7,5)
  \roundnode{V54}(9,5)
  
  \nodetext{V50}(-0.5,0){$v_5$}
 \nodetext{V51}(0.5,0){$v_6$}
 \nodetext{V52}(0.5,0){$v_7$}
 \nodetext{V53}(0.5,0){$v_8$}
  \nodetext{V54}(0.5,0){$v_9$}

     % The sixth level vertices
 \roundnode{V60}(1,3) 
 \roundnode{V61}(3,3)
 \roundnode{V62}(5,3)
 \roundnode{V63}(7,3)
  \roundnode{V64}(9,3)
    
  \nodetext{V60}(-0.5,0){$v_6$}
 \nodetext{V61}(0.5,0){$v_7$}
 \nodetext{V62}(0.5,0){$v_8$}
 \nodetext{V63}(0.5,0){$v_9$}
  \nodetext{V64}(0.5,0){$v_{10}$}

       % The seventh level vertices
 \roundnode{V70}(1,1) 
 \roundnode{V71}(3,1)
 \roundnode{V72}(5,1)
 \roundnode{V73}(7,1)
  \roundnode{V74}(9,1)

  \nodetext{V70}(-0.5,0){$v_7$}
 \nodetext{V71}(0.5,0){$v_8$}
 \nodetext{V72}(0.5,0){$v_9$}
 \nodetext{V73}(0.6,0){$v_{10}$}
  \nodetext{V74}(0.5,0){$v_{11}$}  
  
 %
 % EDGES
 \graphlinewidth{0.025}
 
%  First level edges
  \edge{V0}{V11}
  \edge{V0}{V12}
    \edge{V0}{V13}
      \edge{V0}{V14}
        \edge{V0}{V10}

 % Second level edges

  \edge{V20}{V10}
  \edgetext{V20}{V10}{0}
  \edge{V20}{V11}
  \edgetext{V20}{V11}{1}

    \edge{V21}{V12}
      \edge{V22}{V13}
        \edge{V23}{V14}
          \edge{V24}{V10}    

           % Third level edges

  \edge{V30}{V20}
  \edgetext{V30}{V20}{0}
  \edge{V30}{V21}
  \edgetext{V30}{V21}{1}

    \edge{V31}{V22}
      \edge{V32}{V23}
        \edge{V33}{V24}
          \edge{V34}{V20}    
        
         % Fourth level edges

  \edge{V40}{V30}
  \edgetext{V40}{V30}{0}
  \edge{V40}{V31}
  \edgetext{V40}{V31}{1}

    \edge{V41}{V32}
      \edge{V42}{V33}
        \edge{V43}{V34}
          \edge{V44}{V30}      

         % Fifth level edges

  \edge{V50}{V40}
  \edgetext{V50}{V40}{0}
  \edge{V50}{V41}
  \edgetext{V50}{V41}{1}

    \edge{V51}{V42}
      \edge{V52}{V43}
        \edge{V53}{V44}
          \edge{V54}{V40} 

         % Sixth level edges

  \edge{V60}{V50}
  \edgetext{V60}{V50}{0}
  \edge{V60}{V51}
  \edgetext{V60}{V51}{1}

    \edge{V61}{V52}
      \edge{V62}{V53}
        \edge{V63}{V54}
          \edge{V64}{V50}

% Seventh level edges

  \edge{V70}{V60}
  \edgetext{V70}{V60}{0}
  \edge{V70}{V61}
  \edgetext{V70}{V61}{1}

    \edge{V71}{V62}
      \edge{V72}{V63}
        \edge{V73}{V64}
          \edge{V74}{V60} 

\freetext(-2,13){$h_1(1)= S_0 = 1$}
\freetext(-2,11){$S_1 = 2$}
\freetext(-2,9){$S_2 = 3$}
\freetext(-2,7){$S_3 = 4$}
\freetext(-2,5){$S_4 = 5$}
\freetext(-2,3){$S_5 = 6$}
\freetext(-2,1){$S_6 = 8$}
          
\freetext(1,0.1){\vdots}
\freetext(3,0.1){\vdots}
\freetext(5,0.1){\vdots}
\freetext(7,0.1){\vdots}
\freetext(9,0.1){\vdots}
\end{graph}
\caption{The Bratteli diagram for $S_k = S_{k-1} + S_{\max\{ 0, k-5\}}$}\label{fig:BV}
\end{figure}

Figure~\ref{fig:BV} shows the Bratteli diagram for $Q(k) = \max\{ k-5,0\}$.
In this (stationary) case, the cutting times are
$$
(S_k)_{k \geq 0} = 1,2,3,4,5,6,8,11,15,20,26,34,42,53,\dots
$$
and the incidence matrix is
$$
	F = \begin{pmatrix}
		1 & 1 & 0 & 0 & 0 \\
		0 & 0 & 1 & 0 & 0 \\
		0 & 0 & 0 & 1 & 0 \\
		0 & 0 & 0 & 0 & 1 \\
		1 & 0 & 0 & 0 & 0
	\end{pmatrix}.
$$

Each infinite path has a unique labeling
$(x_i)_{i \geq 1}$ where $x_i$ is the label of the $i+1$-st edge of
the path. For example, the spine represents $\langle 0 \rangle = (0,0,0,\dots) = x_{\min}$ and $\langle N \rangle = \tau^N(x_{\min})$ for all $N \in \N$, so the Vershik map $\tau$ takes the role of the addition of one.
In general, the heights are $h_i(v_i \in V_i) = S_{i-1}$ and $h_i(v_k \in V_i) = S_{Q(k)-1}$.

\begin{prop}\label{prop:qpr}
 If $\sup_k k - Q(k) < \infty$, then the corresponding Cantor system is partially rigid.
\end{prop}

\begin{proof}
Let $d = \sup_k k-Q(k)$. Then telescoping between $2d$ levels will create
a Bratteli diagram with strictly positive incidence matrices, taken from a finite number of possibilities. Therefore the BV-system has exact finite rank. %(and in fact, the corresponding subshift has linear recurrence)
Hence, the result by Danilenko \cite{Danilenko2016} that any exact finite rank system is partially rigid applies.
%
%  A direct argument for the case $Q(k) = \{k-d,0\}$ would be the following:  we  look for a repetition in the edge structure. This gives a lower bound for constant $\alpha$.
% Take any cylinder set $C_N(i)$  of length $N$ ending in $i\in V_N$ {\color{red} [We have notation $C_n(i)$]}, we can decompose it into sets $C_n(1),  \dots, C_n(d)$ where any $C_n(j)$ is the union of subcylinders of $C_N(i)$ ending in vertex $j \in V_n$ for $n\geq N+2(d-1)$. Then parts of $C_n(1)$ will return after $h_n(1)$ steps, as over the next $2(d-1)$ levels there are successive paths between $1 \in V_n$ and  $1 \in V_{n+2(d-1)}$. Thus
% \begin{align*}
% 	\mu(C_N(i) \cap \varphi^{h_n(1)}C_N(i)) 	&\geq \mu(C_n(1) \cap \varphi^{h_n(1)}C_n(1)) \\
% 									&\geq (F^{n-N})_{1,i}\ \mu(C_{n+2(d-1)}(1) = (F^{n-N})_{1,i}\ \frac{\xi_1}{\lambda^{n+2(d-1)-1}} \\
% 									&= \begin{cases} \frac{S_{n-N-d+1}}{\lambda^{n+2(d-1)-N}}\mu(C_N(i)) & \text{ if } i = 1,\\
% 													\frac{S_{n-N-i-d+2}}{\lambda^{n+2(d-1)+d-N-i}(\lambda-1)} \mu(C_N(i)) & \text{ otherwise.}
% 									\end{cases}
% \end{align*}
% Thus by the same Perron-Frobenius argument as before the enumeration system with $Q(k) = \max\{0, k -d \}$ is partially rigid.
\end{proof}

Now we focus on the special case that $Q(k) = \max\{ k-d,0\}$ for $d=1,2,3,\dots$
If $d=1$, then the enumeration scale is isomorphic to the dyadic odometer (and hence even topologically rigid).
% For all $d \geq 2$, the map $T$ is minimal, uniquely ergodic but not invertible as the preimage of $\langle 0 \rangle = 0^\infty$ can be represented in multiple ways, i.e.,
% $$
%  T^{-1}(\langle 0 \rangle) = \{(10^{d-1})^\infty, \ 0(10^{d-1})^\infty, \dots,\ 0^{d-1}(10^{d-1})^\infty\}.
% $$
For $d=2$, the enumeration system is isomorphic to both the Fibonacci substitution shift
and the golden mean Sturmian shift, both well-known to be rigid.
For $d=3$ and $d=4$, the enumeration system is isomorphic to a Pisot
 substitution shift with discrete spectrum, and therefore rigid as well.
See \cite{BruinKellerStPierre1997} for the corresponding computations.
In general, we expect Pisot substitution shifts to be rigid, and indeed, if the Pisot substitution conjecture holds, we obtain a discrete spectrum and hence the required isomorphism to its maximal equicontinuous factor, via the Halmos-von Neumann Theorem\footnote{We don't know if a rigidity proof also exists without these tools (and hence without an answer to the
 Pisot substitution conjecture). We also don't know, whether there are any rigid non-Pisot substitution shifts, see Section~\ref{Sect:OP}.}.

 The interesting case comes when $d=5$, and the enumeration system,
 and its characteristic equation $x^5-x^4-1 = (x-e^{\pi i/3})(x-e^{5\pi i/3})(x^3-x-1) = 0$ are no longer Pisot.

\begin{thm}\label{thm:enum_syst_notrigid}
   The enumeration system $(X,T)$ with $Q(k) = \max\{0, k -5 \}$ is not measure-theoretically rigid.
\end{thm}

\begin{proof}
Using techniques from Bratteli diagrams we can compute the measure of cylinder sets.
% {\color{red} ? Let $\lambda$ be the leading eigenvalue of the transition matrix $F$
% and let $\xi=(\xi_1,\ldots,\xi_d)$ be the corresponding eigenvector with all $\xi_i\geq 0$ and $\sum \xi_i=1$. Note that it is the only eigenvalue of $F$ which is a real number.
% By results of \cite[Theorem~3.8]{BezuglyiKwiatkowskiMedynetsSolomyak2010}
% there exists unique ergodic measure associated with the enumeration system $(X,T)$, and $\xi$ can be used to calculate measures of cylinders. Simply, for any path $e$ of length $n+1$ and ending at vertex $i\in V_{n+1}$ associated cylinder $C_e$ has measure $\mu(C_e)=\frac{\xi_i}{\lambda^n}$.?
% ....
% }
We let $\xi_1, \dots, \xi_5$ be the measures of the cylinder sets of level 1. The structure of the diagram gives that every vertex on level 1 besides the first has a unique outgoing path until it hits the spine on a later level. Therefore
$\xi_i = \frac{\xi_{i-1}}{\lambda}$,
where $\lambda$ is the leading eigenvalue of the transition matrix $F$, see \cite{BezuglyiKwiatkowskiMedynetsSolomyak2010} for the construction. For cylinder sets $C_{n+1}(i)$ of lengths $n+1$ ending in $i-th$-vertex of $V_{n+1}$ the measure is
$$
	\mu(C_{n+1}(1)) = \frac{\xi_1}{\lambda^n} \quad \text{ and } \quad
			\mu(C_{n+1}(i)) = \frac{(\lambda - 1)\xi_1}{\lambda^{n+i-5}} \text{  for $i\not = 1$. }
$$
Furthermore, the number of paths between two levels is connected to $S_k$. After telescoping eight levels, the incidence matrix is full and every vertex of level $N$ connects to every vertex on level $N+8$. The number of paths between the $i$-th vertex in $V_N$ and first in $V_n$ for $n\geq N+8$ is
$$
	(F^{n-N})_{1,i}  = S_{n-N-i-3}.
$$
A simple proof by induction shows that the sequence $(S_{k})_k$
 satisfies the following recursive relation for sufficiently large $k$
 \begin{equation}\label{eq:Sk}
 S_{k+3} = S_{k+2} + S_{k-2} = \begin{cases}
       S_{k+1} + S_k + 1 & \text{ if } k \equiv 0 \text{ or } 5 \bmod 6,\\
       S_{k+1} + S_k - 1 & \text{ if } k \equiv 2 \text{ or } 3 \bmod 6,\\
       S_{k+1} + S_k  & \text{ if } k \equiv 1 \text{ or } 4 \bmod 6.
       \end{cases}
\end{equation}
Define the cylinder set $A = [100000]_{0,5}$ with the digits $100000$ at $x_0 \cdots x_5$ and its subsets
\begin{align*}
    B_k^0 =&\ [100000]_{0,5} \cap [00000000000]_{k-5, k+5},    \\
    B_k^1 =&\ [100000]_{0,5} \cap [00000010000]_{k-5, k+5},     \\
    B_k^{-1} =& \ [100000]_{0,5} \cap [00001000000 ]_{k-5, k+5}.
\end{align*}
These cylinders indicate paths in the Bratteli diagram that pass through the second edge in $E_2$ with
$t(x_0) = v_2 \in V_2$ for the cylinder $A$), and that additionally pass through the second edge in $E_{k+1\pm1}$ with $t(x_{k\pm1}) = v_{k+1\pm1} \in V_{k+1\pm 1}$ (for the cylinders $B_k^{\pm1}$),
while the paths in $B_k^0$ go through the spine between levels $k-5$ and $k+5$.
Note also that $A$, $T(A)$ and $T^{-1}(A)$ are pairwise disjoint, and $B_k^{\pm 1} = T^{S_{k\pm 1}}(B_k^0)$.
The masses of these sets are
\begin{align*}
    \mu(A) &= \frac{\xi_1}{\lambda^2} \quad \text{ and } \quad
    \mu(B_k^0) = \mu(B_k^1) = \mu(B_k^{-1}) = \mu(C_{k+3}(1))(F^{k-7})_{1,1}=\frac{\xi_1 S_{k-11} }{\lambda^{k+2}}.
\end{align*}
By the Perron-Frobenius theorem, for a primitive matrix $F$ and its leading eigenvalue $\lambda$ we have
$$
\lim_{n \rightarrow \infty}\frac{F^n}{\lambda^n} = \xi\eta,
$$
where $\xi$ and $\eta$ are leading right and left strictly positive eigenvectors of $F$ normalised such that $\eta \xi=1$.
Therefore, the fraction $\frac{S_n}{\lambda^n}$ converges to a positive constant and $\mu(B_k^0) = \mu(B_k^1) = \mu(B_k^{-1}) >c>0$ for all $k \in \N$.

To finish the proof, let $n \geq 1$ be arbitrary and take $k = \max\{ j : S_j \leq n\}$ and $n' = n-S_k$, so $n' < S_{k-4}$.

We claim that for at least one of $T^n(B_k^0) \cap A$,
$T^n(B_k^1) \cap A$ and $T^n(B_k^{-1}) \cap A$ has mass $< c/2$,
so $n$ cannot be a rigidity time. This would finish the proof.

To show the claim, take $B' = B_k^0 \cap T^{-n'}(A)$, and note that $T^{n'+S_j}B' \subset A$ for all $j \geq k+2$. Indeed, since $n' < S_{k-4}$, applying $T^{n'}$ to $x \in B'$
can affect at most the first $k-1$ digits. Applying another $T^{S_j}$ doesn't affect any
digit below $k-5$, so $T^{n'+S_j}(x) \in A$.

If $\mu(B') < c/2$, then $\mu(T^{n'}(B^0_k) \setminus A) \geq c/2$ and each $x \in
T^{n'}(B^0_k) \setminus A$ has $0$s in positions $k-2, \dots, k+5$,
but potentially a $1$ in position $k-3$.
Applying another $S_k$ iterates produces a $1$ in position $k$ or $k+1$ 
(since $S_{k+1} = S_k+S_{k-4} = S_k + S_{k-3} - S_{k-8}$).
That is, $T^n(B^0_k \setminus B')$ is disjoint from $B^0_k$ but has mass $\geq c/2$.
Thus the claim holds for $B_k^0$.

Assume therefore that $\mu(B') \geq c/2$. 
\begin{itemize}
 \item $k \equiv 0 \text{ or } 5 \bmod 6$:
 Note that $T^{S_{k+1}}(B_k^0) = B_k^1$ and
 $T^n(B_k^1) = T^{n'+S_{k+1}+S_k}(B_k^0) \supset
 T^{n'+S_{k+3}-1}(B')$, which is a set of mass $\geq c/2$ contained in $T^{-1}(A)$ and hence disjoint from $A$.
 Therefore $\mu(B_k^1 \cap T^n(B_k^1)) < c/2$.
 \item $k \equiv 2 \text{ or } 3 \bmod 6$:
 Note that $T^{S_{k+1}}(B_k^0) = B_k^1$ and
 $T^n(B_k^1) = T^{n'+S_{k+1}+S_k}(B_k^0) \supset
 T^{n'+S_{k+3}+1}(B')$, which is a set of mass $\geq c/2$ contained in $T(A)$ and hence disjoint from $A$.
 Therefore $\mu(B_k^1 \cap T^n(B_k^1)) < c/2$.
 \item $k \equiv 4 \bmod 6$:
 Now $T^{S_{k-1}}(B_k^0) = B_k^{-1}$ and
 $T^n(B_k^{-1}) = T^{n'+S_{k-1}+S_k}(B_k^0) \supset
 T^{n'+S_{k+2}+1}(B')$, which is a set of mass $\geq c/2$ contained in $T(A)$ and hence disjoint from $A$.
 Therefore $\mu(B_k^{-1} \cap T^n(B_k^{-1})) < c/2$.
 \item $k \equiv 1 \bmod 6$:
 Now $T^{S_{k-1}}(B_k^0) = B_k^{-1}$ and
 $T^n(B_k^{-1}) = T^{n'+S_{k-1}+S_k}(B_k^0) \supset
 T^{n'+S_{k+2}-1}(B')$, which is a set of mass $\geq c/2$ contained in $T^{-1}(A)$ and hence disjoint from $A$.
 Therefore $\mu(B_k^{-1} \cap T^n(B_k^{-1})) < c/2$.
\end{itemize}
This finishes the proof of the claim and of the whole theorem.
\end{proof}

\medskip
\section{Open problems}\label{Sect:OP}

\begin{enumerate}
    \item Are there partially rigid but not measure-theoretically rigid Toeplitz systems? Can one use the construction from Theorem~\ref{thm:skew} with another subshift $(Y, \sigma, \nu)$ to get a partially rigid, but not measure-theoretically rigid Toeplitz system?

    \item Suppose a minimal Cantor system $(X,T,\mu)$ is measure-theoretically rigid (or partially rigid). Does it mean that the first return time map to any clopen set of positive measure is also measure-theoretically rigid (or partially rigid)?

    \item Compute the best partial rigidity constant (in \cite{DonosoMaassRadic2023} this constant is called a ``partial rigidity rate'') for enumeration systems considered in Section~\ref{sec:enum_systems}.
    %defined by linear recursion. 
    Determine which enumeration systems 
    defined by linear recursion are measure-theoretically rigid.

    \item If a primitive substitution shift is rigid, does it mean that its leading eigenvalue is Pisot? (Note that this is less restrictive than that the substitution itself is Pisot. For example, constant length substitutions are rigid by Proposition~\ref{prop:length_reg_Toeplitz} and their leading eigenvalues are integer, hence Pisot, even though the incidence matrix may have other eigenvalues outside the unit disk.)

    %\item If a topological dynamical system has uncountably many rigid ergodic invariant probability measures, does it follow that every invariant probability measure is partially rigid? 
    %Donoso homeomorphism of a cmpct metric space

    \item It is known that every Cantor minimal system is strongly orbit equivalent to a Cantor minimal system of entropy zero (see \cite{BoyleHandelman1994}, \cite{Durand2010}).
    Is it true that every Cantor minimal system has in its strong orbit equivalence class a Cantor minimal system which is (partially) measure-theoretically rigid with respect to at least one of its ergodic invariant probability measures (see also Remark \ref{Rem:SOEtoNonRigid})? In particular, is it true that any simple Bratteli diagram can be telescoped to a diagram admitting an order such that the corresponding Vershik map is rigid with respect to at least one of its ergodic invariant probability measures?

\end{enumerate}

%joint rigidity sequence

\medskip
\textbf{Acknowledgements.}
HB, OK and SR are supported by the Polish National Agency for Academic Exchange under the contract no. PPN/BAT/2021/1/00024/U/00001.
OK is supported by the program ``Excellence Initiative - Research University'' at the AGH University of Krak\'ow. PO  is supported by National Science Centre, Poland (NCN), grant no. 2019/35/B/ST1/02239.
HB is supported by the ANR-FWF Project I 6750-N. 
This research was finalized during the visit of HB and SR to the Jagiellonian
University funded
 by the program ``Excellence Initiative - Research University'' at the
Jagiellonian University in Krak\'ow.

\bibliography{rigidity-bib}
\bibliographystyle{alpha}

%cylinder sets $C_v(n)$
%cylinder measures $p_v(n)$
%height $h_v(n)$
%Vershik map $\varphi_B$
\end{document}